\documentclass[11pt]{amsart}
\usepackage{geometry}                
\usepackage[latin1]{inputenc}   
\usepackage{mathpazo,cite}  
\usepackage[usenames,dvipsnames,cmyk]{xcolor}
\usepackage{url}
\usepackage{amsmath,amssymb,latexsym,mathrsfs,amssymb,amsthm,amsfonts}
\usepackage{enumerate}
\usepackage{enumitem,overpic}
\usepackage{mathtools,mathdots} 
\usepackage[all]{xy}
\usepackage{tikz, tkz-euclide}
\usepackage{setspace}
\renewcommand{\vec}{\mathbf}
\newcommand{\lvertiii}{\left\vert\kern-0.25ex\left\vert\kern-0.25ex\left\vert}
\newcommand{\rvertiii}{\right\vert\kern-0.25ex\right\vert\kern-0.25ex\right\vert}

\usepackage[colorlinks=true, linkcolor=MidnightBlue, citecolor=MidnightBlue, urlcolor=MidnightBlue]{hyperref}

\newtheorem{theorem}{Theorem}[section]
\newtheorem{lemma}[theorem]{Lemma}
\newtheorem{proposition}[theorem]{Proposition}
\newtheorem{corollary}[theorem]{Corollary}

\theoremstyle{definition}
\newtheorem{definition}[theorem]{Definition}

\theoremstyle{remark}
\newtheorem{remark}[theorem]{Remark}

\DeclareMathOperator{\ran}{ran}
\DeclareMathOperator{\id}{Id}
\newcommand*\dd{\mathop{}\!\mathrm{d}}
\newcommand{\ii}{\ensuremath{\mathrm{i}}}
\newcommand{\ee}{\ensuremath{\,\mathrm{e}}}

\title{On the convergence of Fourier spectral methods involving non-compact operators}

\author{Thomas Trogdon}
\address{Department of Applied Mathematics, University of Washington, WA}
\email{trogdon@uw.edu}
\urladdr{http://faculty.washington.edu/trogdon/}

\begin{document}

\begin{abstract}
    Motivated by Fredholm theory, we develop a framework to establish the convergence of spectral methods for operator equations $\mathcal L u = f$.  The framework posits the existence of a left-Fredholm regulator for $\mathcal L$ and the existence of a sufficiently good approximation of this regulator. Importantly, the numerical method itself need not make use of this extra approximant.  We apply the framework to Fourier finite-section and collocation-based numerical methods for solving differential equations with periodic boundary conditions and to solving Riemann--Hilbert problems on the unit circle.  We also obtain improved results concerning the approximation of eigenvalues of differential operators with periodic coefficients.
\end{abstract}

\maketitle

\section{Introduction}

The classical numerical analysis fact that the combination of consistency and stability imply convergence is a powerful one.  Consistency plays a key role in developing methods and is usually straightforward to establish. Stability, on the other hand, is often quite difficult.  When developing numerical methods for linear operator equations, establishing stability requires either a general theory or a detailed analysis of the linear system that results.  For example, when solving
\begin{align} \label{eq:bvp}
\begin{cases}
    u''(x) = f(x),\\
    u(0) = \alpha,\\
    u(1) = \beta,
\end{cases}
\end{align}
with a centered finite difference, the linear system $\mathcal L_N u_N = f_N$ that results can be fully analyzed: one uses Chebyshev polynomials to compute the eigenvalues of the discretized operator explicitly \cite{leveque}.  On the other hand, for second-kind Fredholm integral equations of the form
\begin{align} \label{eq:K}
    u(x) - \int_0^1 K(x,y) u(y) \dd y= f(x),
\end{align}
collocation methods using Gauss-Legendre nodes and weights fit in an abstract framework \cite{atkinson,Kress1993,Kress2014} because the operator
\begin{align*}
    u \mapsto \int_0^1 K(x,y) u(y) \dd y,
\end{align*}
is compact for smooth kernel functions $K$. But what is one to do if the problem at hand does not fit into either category?  One wants to consider both (i) problems general enough where one does not have an explicit handle on the spectrum (and eigenvector matrix condition number) of the discretization and (ii) problems that do not come from the discretization of a compact perturbation of the identity.  In this paper, we propose a framework to help continue to fill this void.  The main idea we present here is that if one can find an appropriate preconditioner, not necessarily to use for numerical purposes, convergence follows.  Informally, if one aims to solve an operator equation via a finite-dimensional approximation
\begin{align*}
    \mathcal L u =f  \quad \Longrightarrow \quad \mathcal L_N u_N = f_N,
\end{align*}
and one can identify an operator $\mathcal N$ and a finite-dimensional approximation $\mathcal N_N$ such that
\begin{align*}
    \mathcal N \mathcal L = \id - \mathcal K, \quad \mathcal N_N \mathcal L_N = \id - \mathcal K_N, \quad \|\mathcal K_N - \mathcal K\| \to 0,
\end{align*}
then convergence follows, $u_N \to u$.  The two extremes of employing this framework are:
\begin{itemize}
    \item In \eqref{eq:bvp}, let $\mathcal N$ be the Green's function solution operator and let $\mathcal N_N$ be the inverse of $\mathcal L_N$.  Then $\mathcal K, \mathcal K_N = 0$.
    \item In \eqref{eq:K}, let $\mathcal N, \mathcal N_N = \id$, $\mathcal K u(x) = \int_0^1 K(x,y) u(y)$ and $\mathcal K_N$ is its direct discretization.
\end{itemize}
But many of the operator equations one runs into in practice fit somewhere between these two idealized settings, and the proposed framework addresses some of these cases.

We illustrate the ideas with two core examples: solving differential equations with periodic boundary conditions and solving Riemann--Hilbert problems on the circle. In the former case, we include a significant detour to demonstrate that this idea can also be used to show convergence of eigenvalues.  Thus, we obtain some improvements on the results in \cite{Curtis,Zumbrun} for the convergence of Hill's method \cite{hill} for both self-adjoint and non-self-adjoint operators.  Indeed, it is the left preconditioning step in \cite{Zumbrun} (see also \cite{Krasnoselskii1972} and \cite{Olver2013} for a right preconditioning step) where the authors regularize their problem to find a 2-modified Fredholm determinant --- the Evans function --- that motivated this work. 

The primary goal of this work is not to present the convergence theory for numerical methods that have not yet been shown to converge, but rather to unify, and hopefully simplify, the theory for two classes of operator equations that typically require a very different treatment (see \cite{Krasnoselskii1972} for differential operators and \cite{prossdorf} for singular integral operators).

We should also point out the idea that once one has a method to solve an operator equation, one may have a method to approximate eigenvalues (and maybe eigenvectors too) is far from new.  There is a lengthy text on this matter \cite{Chatelin2011} and this is the approach pursued in \cite{Krasnoselskii1972}.  The work here could surely be put into these frameworks, but since we obtain spectrum estimates rather directly, we do not attempt this here.  Furthermore, many modern methods can use a convergent discretization as input to compute spectral quantities \cite{Colbrook2021}.

The two examples we present are based on Fourier (Laurent) series.  Some of the calculations we present are eased through the elementary nature of such series (following \cite{Kress1993}).  Specifically, the classical projection and interpolation operators are easily analyzed on Sobolev spaces.  Further, the differential and singular integral operators we consider act simply on the Fourier basis and hence on these Sobolev spaces. Yet, it is only when we consider the relation of the Fourier interpolation projection to the orthogonal projection do we need to examine the precise form of the linear system that results.  It is because of this that we believe these ideas will be of use far beyond the setting of Fourier series, and this is an area of ongoing research.

\subsection{Notation}
Before we proceed, we lay out some notation:
\begin{itemize}
    \item $\mathbb U = \{ z \in \mathbb C : |z| = 1\}$ with counter-clockwise orientation.
    \item $\mathbb T = [0, 2\pi)$ with $0$, $2\pi$ identified.
    \item Scalar-valued functions of a continuous variable are denoted by lower-case Latin letters, e.g., $u$.
    \item Vectors with scalar entries are denoted using lower-case bold Latin letters, e.g., $\vec u$.  And typically, $\vec u$ will be understood to be the expansion coefficients of $u$ in some basis that is clear from context.  The entries in the vector $\vec u$ will be referred to by $u_j$.
    \item If a vector $\vec u_N$ has a subscript, then we will use $u_{N,j}$ to refer to its entries.
    \item We will use $\|\cdot \|_s$ to denote an $H^s$ Sobolev norm.  Then $\|\cdot\|_{s \to t}$ will be used to denote the induced operator norm from $H^s$ to $H^t$.  If $s = t$ we will simply use $\|\cdot\|_s$ to denote both the norm and induced operator norm.
    \item $\ran \mathcal L$ will be used to denote the range of an operator $\mathcal L$ where the domain for $\mathcal L$ is implicit.
    \item $\id$ is used to denote the identity operator.
\end{itemize}

\subsection{Overview of main results}

The main developments begin with Definition~\ref{def:main} where an admissible operator is defined.  Then the main abstract result is Theorem~\ref{t:main} that uses an additional hypothesis that the Fredholm regulator has an appropriate finite-dimensional approximation.  This result and its extension in Theorem~\ref{t:IN-PN} give sufficient conditions for the convergence of a numerical method for operator equations that encompass the classical settings discussed above.

In Section~\ref{sec:period}, we discuss Sobolev spaces on which our operator equations will be posed.  We also discuss the Fourier orthogonal projection and interpolation operators.  Both of our main applications will use these spaces.  The use of the orthogonal projection gives what is often known as the finite-section method.  The use of the interpolation operators gives collocation methods.

In Section~\ref{sec:diffop}, we discuss our first application of the framework to the solution of differential equations on periodic spaces of functions.   The operators we consider are assumed to have their highest-order terms have constant coefficients:
\begin{align}\label{eq:diff_op_simple}
    \mathcal L u (\theta) = \sum_{j=q}^k c_j \frac{\dd^j u}{\dd \theta^j}(\theta) +\sum_{j=0}^{p} a_j(\theta)\frac{\dd^j u}{\dd \theta^j}(\theta), \quad p < k.
\end{align}
See \eqref{eq:diff_op} for more detail.  The first  result for this operator is Theorem~\ref{t:diff-conv} which gives sufficient conditions for the convergence of the so-called finite-section method at an optimal rate \eqref{eq:diff-conv}.  It is noted that the same arguments hold if one instead uses the interpolation projection.  Thus, convergence follows for collocation methods.  And if the coefficients are infinitely smooth, convergence occurs spectrally --- faster than any polynomial rate:
\begin{align*}
    \|u - u_N\|_s \leq C_{s,t} N^{s-t},
\end{align*}
for any $t >  s$.

We then use this framework to analyze the convergence of the spectrum of operators of the form \eqref{eq:diff_op_simple} (see also \eqref{eq:diff_op}).  The main motivation here is to the application of our results to the Fourier--Floquet--Hill method \cite{hill}.  The general result is given in Theorem~\ref{t:eval-gen} which establishes, in the general setting, that : (i) For any $\epsilon > 0$, all eigenvalues of the finite-dimensional approximation eventually enter the $\epsilon$-pseudospectrum\footnote{Note a particular definition of the pseudospectrum, see \eqref{eq:pseudo}, is used where the operator is mapping between two different spaces.  As a result, this does not imply uniform approximation of the spectrum.  } of the infinite-dimensional operator. And (ii) every eigenvalue of the infinite-dimensional operator is the limit of eigenvalues of the finite-dimensional approximations.  A common way to discuss such convergence is to state that it gives convergence in the Attouch-Wetts metric \cite{Colbrook2022}.

If more is known about the operator in question, rates of convergence may be determined. 
 So, in Section~\ref{sec:sa}, we present significant improvements of these results for self-adjoint operators. Specifically, we establish a general result (Theorem~\ref{t:sa-mult}) on the preservation of multiplicities.  Surely, these techniques could also be used to show the convergence of the eigenfunctions, but this is not an issue we take up here. Then Theorem~\ref{t:sa-mult} gives a result on the convergence rate of the approximations of the eigenvalues.  When these results are applied to the operator in \eqref{eq:diff_op_simple} (see also \eqref{eq:diff_op}), in the self-adjoint setting, we arrive at Theorem~\ref{t:spectrum-main} which gives the spectral convergence of the eigenvalues for smooth coefficient functions:
\begin{align*}
    |\lambda - \lambda_N| \leq C_t(\lambda) N^{-t}.
\end{align*}

In Section~\ref{sec:rhp}, we shift focus to Riemann--Hilbert problems on the unit circle: Given $g: \mathbb U \to \mathbb C$, find $\phi: \mathbb C \setminus \mathbb U \to \mathbb C$, analytic on its domain of definition, such that its boundary values $\phi^\pm$ on $\mathbb U$ satisfy
\begin{align}\label{eq:jump}
    \phi^{+}(z) = \phi^-(z) g(z), \quad \phi(\infty) = 1.
\end{align}
Generally speaking, Riemann--Hilbert problems are boundary-value problems in the complex plane for sectionally analytic functions.  We treat, arguably, the most elementary form of such a problem that is not without its complications.  Indeed, the complication stems from the fact that the operator one considers is a non-compact perturbation of the identity, complicating proofs of convergence.  Our proof leans on an important result of Gohberg and Feldman (see Theorem~\ref{t:G-F}).  The numerical method for this problem that we analyze is of particular interest because of its wide applicability to much more general Riemann--Hilbert problems, see \cite{Ballew2023,TrogdonSOBook,SORHFramework,SOHilbertTransform}, for example.  The main result establishes that
\begin{align*}
    \|u - u_N\|_s \leq C_{s,t} N^{s-t},
\end{align*}
for any $t >  2s$.  We conjecture that the requirement of $2s$ is suboptimal, and could, potentially, be replaced with $s$ if Theorem~\ref{t:G-F} is extended to Sobolev spaces.  Yet, the convergence rate is optimal.

\subsection{Relation to previous work}

The preconditioning, left-Fredholm regulator ideas that underpin Definition~\ref{def:main} were motivated by \cite{Zumbrun}.  We note that a right preconditioning step appears in \cite[Theorem 15.4]{Krasnoselskii1972} for a relatively compact perturbation of an unbounded operator and in \cite{Olver2013} for a Petrov--Galerkin ultraspherical polynomial method.  The hypotheses of \cite[Theorem 15.4]{Krasnoselskii1972} also include a commuting assumption that is similar to one we make about the projection operators.  This theorem,  \cite[Theorem 15.4]{Krasnoselskii1972}, applies in our setting, giving similar results as in our Theorem~\ref{t:diff-conv} for differential operators, but it does not appear to apply to the Riemann--Hilbert problems considered in Section~\ref{sec:rhp} without further modifications.

As noted above, in \cite{Zumbrun}, the authors use a 2-modified Fredholm determinant to show convergence to eigenvalues of the operator \eqref{eq:diff_op}.  Rates are not obtained in this paper, so our results are an improvement in this respect.  But the results in \cite{Zumbrun} apply to systems of differential equations and allow for the leading term to have a non-constant coefficient, provided that multiplication by this function is symmetric, positive definite.  It remains an interesting question to extend the current work to this setting, and obtain rates. The extension of the current work to systems of equations is straightforward.

We also point to the comprehensive articles \cite{Hansen2008,Ben-Artzi2015a,Ben-Artzi2015} for a survey of computational aspects of spectral theory. Hansen \cite{Hansen2008} gives an overview of what is known for the approximation of spectra of primarily self-adjoint and normal operators on Hilbert space, going beyond operators with pure point spectrum, see, for example, \cite[Theorem 15]{Hansen2008}. We use the classical square truncations of operators approximations in this work.  But we point out that rectangular truncations are also of use in spectral compuations, and can give better guarantees \cite{Colbrook2019a}.

In the self-adjoint case, our results are an improvement upon those in \cite{Curtis} because (1) we obtain spectral convergence for all eigenvalues, not just extremal ones and (2) we allow operators of a more general form.  On the other hand, \cite{Curtis} establishes the convergence of eigenfunctions, something we do not exactly address.  Although, we do show, in a restricted sense, the convergence of eigenprojections within the proof of Theorem~\ref{t:eval-gen}.  Many of these results can be deduced from \cite{Krasnoselskii1972}, but we present different, direct proofs of these facts in a framework that then applies to the operators in Section~\ref{sec:rhp}. 

There has also been a flurry of work putting spectral computations within the Solvability Complexity Index (SCI) hierarchy \cite{Colbrook2022,Hansen2010,Colbrook2019}.  The questions asked and answered in these works, while closely related to the current, have a different focus.  The SCI hierarchy is primarily focused on the existence or non-existence of algorithms with a given complexity, measured by the number of limits required to compute a quantity.  Often, the finite-section method, very similar to what we discuss here, is used to produce an algorithm ---- giving existence.  The convergence rate of such an algorithm may be of secondary importance in these works, but it is certainly not ignored \cite{Colbrook2021,Colbrook2021a}. In the current work, convergence rates are a main focus.

Regarding the results in Section~\ref{sec:rhp}, similar results can be found in the work of Pr\"ossdorf and Silberman \cite[Chapters 8 and 9]{prossdorf} for H\"older spaces, instead of Sobolev spaces.  The problem we pose is equivalent to the classical Weiner--Hopf factorization of a symbol \cite{Bottcher1997}. As stated, it has an explicit solution --- one possible reason for a relative lack of literature on the convergence theory for methods.  Nevertheless, the wide applicability (again, see, \cite{Ballew2023,TrogdonSOBook,SORHFramework,SOHilbertTransform}) of the ideas beg for the simplest case to be well understood in a framework that generalizes.  Furthermore, the methods we propose appear to apply in the case where $g$ in \eqref{eq:jump} is matrix-valued, as is $\phi$.  There is no longer, in general, an explicit solution to this Riemann--Hilbert problem.

Many other results exist in the literature for the convergence of numerical methods for weakly singular and strongly singular integral equations.  The classic paper of Kress and Sloan \cite{Kress1993} presents a problem that resembles our setting in that they consider an operator of the form $\mathcal L + \mathcal A$ where the discretization of $\mathcal L$ can be directly checked to be invertible and $\mathcal A$ is compact.  Such an operator would likely fit within the framework we propose here (see also \cite{Kress1995}). 

\section{The classical framework}
We consider an operator equation
\begin{align}\label{eq:op-eqn}
    \mathcal L u = f,
\end{align}
where $\mathcal L$ is a bounded linear operator $\mathcal L: \mathbb V \to \mathbb W$ and $\mathbb V, \mathbb W$ are separable Banach spaces with $f \in \mathbb W$.  
Suppose we have access to two finite-dimensional projection operators $\mathcal P_N^{\mathbb V}$, $\mathcal P_N^{\mathbb W}$ such that
\begin{align*}
    \|P_N^{\mathbb V} v - v\|_{\mathbb V} \to 0, \quad \|P_N^{\mathbb W} w - w\|_{\mathbb W} \to 0,
\end{align*}
as $N \to \infty$ for any fixed $v \in \mathbb V, w \in \mathbb W$.   Note that even in the case $\mathbb W = \mathbb V$ we may have that $\mathcal P_N^{\mathbb V} \neq \mathcal P_N^{\mathbb W}$. 

We then replace \eqref{eq:op-eqn} with
\begin{align}\label{eq:disc-eqn}
    \mathcal P_N^{\mathbb W} \mathcal L u_N = \mathcal P_N^{\mathbb W} f, \quad u_N \in \ran \mathcal P_N^{\mathbb V}.
\end{align}
This equation may or may not be solvable, and even if it is, one may not have $u_N \to u$ in any meaningful sense.  We first explain classical conditions that one can impose to guarantee $u_N \to u$ in $\mathbb V$.

\subsection{The classical framework}

The classical approach to understanding the approximation of solutions of operator equations is to consider operator equations of the form
\begin{align}
    (\id - \mathcal K) u = f,
\end{align}
where $\mathcal K$ is a \emph{compact} bounded linear operator $\mathcal K: \mathbb V \to \mathbb W$ and $f \in \mathbb W$.  
The first important result is the following lemma \cite{atkinson}.

\begin{lemma}
Suppose that 
\begin{align*}
\|\mathcal P_N^{\mathbb W} u - u \|_{\mathbb W} \overset{N \to\infty}{ \to } 0 ,
\end{align*}
for any $u \in \mathbb W$.  Then if $\mathcal K : \mathbb V \to \mathbb W$ is compact
\begin{align*}
    \|\mathcal P_N^{\mathbb W} \mathcal K - \mathcal K\|_{\mathbb V \to \mathbb W} \overset{N \to\infty}{ \to } 0 .
\end{align*}
\end{lemma}
To more simply describe the classical setting, we now set $ \mathcal P_N^{\mathbb W} =  \mathcal P_N^{\mathbb V} =  \mathcal P_N$, $\mathbb W = \mathbb V$.    With the lemma in hand, the key approximation result can be established \cite{atkinson}.  We include the proof because the arguments will be modified below.
\begin{theorem}\label{t:classic-approx}
Suppose that $\mathcal K : \mathbb V \to \mathbb V$ is compact, $\id - \mathcal K$ is invertible, and that
\begin{align*}
    \| \mathcal P_N v - v\|_{\mathbb V} \overset{N \to \infty}{\to} 0,
\end{align*}
for every fixed $v \in \mathbb V$.  Then there exists $N_0$ such
\begin{align*}
     \|\mathcal P_N \mathcal K - \mathcal K\|_{\mathbb V} < \frac{1}{\|(\id - \mathcal K)^{-1}\|_{\mathbb V}}, \quad N > N_0,
\end{align*}
and therefore $\id - \mathcal P_N \mathcal K$ is invertible for all $N > N_0$.  Furthermore, there exists constants $c < C$, independent of $N$, such that
\begin{align*}
    c\| \mathcal P_N u -u \|_{\mathbb V} \leq \| u_N -u \|_{\mathbb V} \leq C \| \mathcal P_N u -u \|_{\mathbb V}, \quad u = (\id - \mathcal K)^{-1} f, \quad N > N_0,
\end{align*}
where $u_N$ is the unique solution of \eqref{eq:disc-eqn} and
\begin{align*}
    C \leq   \max_{N > N_0} \frac{\| (\id - \mathcal K)^{-1}\|_{\mathbb V}}{1 - \| (\id - \mathcal K)^{-1}\|_{\mathbb V}\|\mathcal P_{N} K - \mathcal K\|_{\mathbb V}}.
\end{align*}
\end{theorem}
\begin{proof}
Since $(\id - \mathcal K)$ is invertible, we consider the operator
\begin{align*}
   ( \id - \mathcal K )^{-1}  (\id - \mathcal P_N \mathcal K) & =  ( \id - \mathcal K )^{-1} ( \id - \mathcal K + (\mathcal K-  \mathcal P_N \mathcal K)\\
   & = \id - ( \id - \mathcal K )^{-1}(\mathcal P_N \mathcal K -\mathcal K).
\end{align*}
Then there exists $N_0$ such that if $N > 0$ then
\begin{align*}
    \|\mathcal P_N \mathcal K - \mathcal K\|_{\mathbb V} < \frac{1}{\|(\id - \mathcal K)^{-1}\|_{\mathbb V}},
\end{align*}
in which case it follows that $\id - \mathcal P_N \mathcal K$ is invertible for $N > N_0$ and
\begin{align*}
    \|(\id - \mathcal P_N \mathcal K)^{-1}\|_{\mathbb V} \leq  \frac{\|(\id - \mathcal K)^{-1}\|_{\mathbb V}}{1 - \|( \id - \mathcal K )^{-1}\|_{\mathbb V} \|\mathcal P_N \mathcal K -\mathcal K\|_{\mathbb V}}.
\end{align*}
Then we may consider
\begin{align*}
    (\id - \mathcal K) u &= f,\\
    (\mathcal P_N - \mathcal P_N \mathcal K )u &= \mathcal P_N f,\\
    (\id - \mathcal P_N \mathcal K)u &= u - \mathcal P_N u + \mathcal P_N f,\\
    (\id - \mathcal P_N \mathcal K)u &= u - \mathcal P_N u + u_N - \mathcal P_N \mathcal K u_N,\\
    (\id - \mathcal P_N \mathcal K) (u - u_N) &= u - \mathcal P_N u.
\end{align*}    
This implies that there exists constants $c < C$ such that if $N > N_0$ then 
\begin{align*}
    c\| \mathcal P_N u -u \|_{\mathbb V} \leq \| u_N -u \|_{\mathbb V} \leq C \| \mathcal P_N u -u \|_{\mathbb V}.
\end{align*}
\end{proof}
While this is an elegant result, its usefulness is limited by the fact that many operators one needs to consider are not compact perturbations of the identity.  Also, interpolation projection operators often do not converge strongly to the identity.  They only converge on a dense subspace.

\section{Extending the classical framework -- left-Fredholm operators}

We begin with a definition that we use as motivation.
\begin{definition}
A bounded linear operator $\mathcal L: \mathbb V \to \mathbb W$ is said to be left Fredholm if there exists another bounded linear operator $\mathcal N: \mathbb W \to \mathbb V$ such that
\begin{align*}
    \mathcal N \mathcal L = \id - \mathcal K,
\end{align*}
where $\mathcal K: \mathbb V \to \mathbb V$ is compact.  The operator $\mathcal N$ is said to be a left Fredholm regulator for $\mathcal L$.
\end{definition}
Fredholm operators can be thought of as, in a sense, compactifiable operators.  This may appear to be trivial as all operators we will consider are left Fredholm --- operators with left inverses are left Fredholm.  Yet this concept motivates important properties we look for that allow us to prove convergence for the discretizations of operator equations.

\begin{remark}
    From a computational perspective, identifying a sufficiently simple Fredholm regulator for an unbounded operator can give a principled preconditioner for iterative methods, as is commonly required for the efficiency of classical finite difference and finite element methods.
\end{remark}

\begin{definition}\label{def:main}
A bounded linear operator $\mathcal L: \mathbb V \to \mathbb W$ is admissible if
\begin{itemize}
    \item $ \mathcal L = \mathcal L_0 + \mathcal L_1$, where $\mathcal L_0 \mathcal P_N^{\mathbb V} = \mathcal P_N^{\mathbb W} \mathcal L_0$,
    \item $\mathcal L$ is left Fredholm with regulator $\mathcal N$,
    \item $\mathcal N = \mathcal N_0 + \mathcal N_1$, where $\mathcal N_0 \mathcal P_N^{\mathbb W} = \mathcal P_N^{\mathbb V} \mathcal N_0$.
\end{itemize} 
\end{definition}
The importance of a relation such as $\mathcal L = \mathcal L_0 + \mathcal L_1$, where $\mathcal L_0 \mathcal P_N^{\mathbb V} = \mathcal P_N^{\mathbb W} \mathcal L_0$, is that if one wants to solve
\begin{align}\label{eq:1}
    \mathcal P_N^{\mathbb W} \mathcal L u_N = \mathcal P_N^{\mathbb W} f, \quad u_N \in \ran \mathcal P_N^{\mathbb V},
\end{align}
then it is equivalent to consider
\begin{align}\label{eq:2}
    ( \mathcal L_0 + \mathcal P_N^{\mathbb W} \mathcal L_1) u_N = \mathcal P_N^{\mathbb W} f, \quad u_N \in \ran \mathcal P_N^{\mathbb V},
\end{align}
Indeed if $u_N$ solves \eqref{eq:1} then because $\mathcal P_N^{\mathbb V}u_N = u_N$ we find that $u_N$ solves \eqref{eq:2}.  Similarly, if $u_N$ solves \eqref{eq:2} then, again using $\mathcal P_N^{\mathbb V}u_N = u_N$ we have that $u_N$ solves \eqref{eq:1}.

\begin{remark}
    The restriction $\mathcal N_0 \mathcal P_N^{\mathbb W} = \mathcal P_N^{\mathbb V} \mathcal N_0$ may seem significant.  This type of assumption can be found in a related form in \cite{Krasnoselskii1972,Kress1993} and other works.  Replacing it with an approximate condition $\mathcal N_0 \mathcal P_N^{\mathbb W} \approx \mathcal P_N^{\mathbb V} \mathcal N_0$ is an area of ongoing research. The convergence result in \cite{Olver2013}, for example, after some modifications, appears to fit into this framework.
\end{remark}

The following result is now immediate.  While fairly elementary, its usefulness is demonstrated via the applications that follow. It is applied to situations where $\mathcal L_1$ is a non-compact operator and the classical framework fails to apply.

\begin{theorem}\label{t:main}
Let $\mathcal L = \mathcal L_0 + \mathcal L_1$ be admissible and invertible with invertible left-Fredholm regulator $\mathcal N = \mathcal N_0 + \mathcal N_1$.  Suppose there exists operators $\mathcal N_{1,N}: \ran \mathcal P_N^{\mathbb W} \to \ran \mathcal P_N^{\mathbb V}$ and $\mathcal K_N :\mathbb V \to \mathbb V$ such that
\begin{align}\label{eq:comp}
    (\mathcal N_0 + \mathcal N_{1,N})(\mathcal L_0 + \mathcal P_N^{\mathbb W}\mathcal L_1) = \id - \mathcal K_N, \quad \text{on} \quad \ran \mathcal P_N^{\mathbb V},
\end{align}
where
\begin{align*}
    \| \mathcal K_N - \mathcal K \|_{\mathbb V} \overset{N \to \infty}{\to} 0, \quad \mathcal N \mathcal L = \id - \mathcal K, \quad \mathcal K \text{ compact.}
\end{align*}
Then if $\dim \ran \mathcal P_N^{\mathbb V} = \dim \ran \mathcal P_N^{\mathbb W}$, there exists $N_0 > 0$ such that for $N > N_0$, $\mathcal L_0 + \mathcal P_N^{\mathbb W}\mathcal L_1: \ran \mathcal P_N^{\mathbb V} \to \ran \mathcal P_N^{\mathbb W}$ is invertible and for $f \in \mathbb W$
\begin{align*}
    (\id - \mathcal K_N) (u - u_N) &= (\mathcal N_0 + \mathcal N_{1,N})\mathcal L_0 (u - \mathcal P_N^{\mathbb V}u),\\
    u = \mathcal L^{-1} f, \quad u_N &= (\mathcal L_0 + \mathcal P_N^{\mathbb W}\mathcal L_1)^{-1} \mathcal P_{N}^{\mathbb W} f.
\end{align*}
Therefore
\begin{align*}
    \| u - u_N\|_{\mathbb V} \leq C_N \| u - \mathcal P_N^{\mathbb V} u\|_{\mathbb V},
\end{align*}
where
\begin{align*}
    C_N = \frac{\|(\id - \mathcal K)^{-1}\|_{\mathbb V}}{1 - \|(\id - \mathcal K)^{-1}\|_{\mathbb V} \|\mathcal K - \mathcal K_N\|_{\mathbb V}}\|(\mathcal N_0 + \mathcal N_{1,N})\mathcal L_0\|_{\mathbb V}.
\end{align*}
\end{theorem}

\begin{proof}
The assumptions imply that $\id - \mathcal K_N: \mathbb V \to \mathbb V$ is invertible for sufficiently large $N$, say $N > N_0$.  And therefore, it is invertible as an operator on $\ran \mathcal P_N^{\mathbb V}$. Indeed, since $\ran \mathcal P_N^{\mathbb V}$ is an invariant subspace for $\id - \mathcal K_N$, if it failed to be invertible on this subspace, then it must have a non-trivial nullspace.  That would imply the same for $\id - \mathcal K_N$. This implies both operators in \eqref{eq:comp} are invertible as finite-dimensional operators between $\ran \mathcal P_N^{\mathbb V}$ and $\ran \mathcal P_N^{\mathbb W}$.  This establishes the first claim.  Then consider

\begin{align*}
    (\mathcal L_0 + \mathcal L_1) u &= f,\\
    (\mathcal P_N^{\mathbb W} \mathcal L_0 + \mathcal P_N^{\mathbb W}\mathcal L_1) u &= \mathcal P_N^{\mathbb W} f,\\
    (\mathcal L_0 + \mathcal P_N^{\mathbb W}\mathcal L_1) u &= \mathcal L_0 (u - \mathcal P_N^{\mathbb V}u) + \mathcal P_N^{\mathbb W} f,\\
   (\mathcal L_0 + \mathcal P_N^{\mathbb W}\mathcal L_1) u &= \mathcal L_0 (u - \mathcal P_N^{\mathbb V}u) + \mathcal L_0 u_N + \mathcal P_N^{\mathbb W} \mathcal L_1 u_N,\\
    (\mathcal L_0 + \mathcal P_N^{\mathbb W}\mathcal L_1) (u - u_N) &= \mathcal L_0 (u - \mathcal P_N^{\mathbb V}u),\\
    (\id - \mathcal K_N) (u - u_N) &= (\mathcal N_0 + \mathcal N_{1,N})\mathcal L_0 (u - \mathcal P_N^{\mathbb V}u).
\end{align*}
Therefore
\begin{align*}
    \|u - u_N\|_{\mathbb V} \leq \|(\id - \mathcal K_N)^{-1}\|_{\mathbb V} \|(\mathcal N_0 + \mathcal N_{1,N})\mathcal L_0\|_{\mathbb V} \|u - \mathcal P_N^{\mathbb V}u\|_{\mathbb V}.
\end{align*}
And the final claim follows from bounding $\|(\id - \mathcal K_N)^{-1}\|_{\mathbb V}$ in terms of $\|(\id - \mathcal K)^{-1}\|_{\mathbb V}$.
\end{proof}
\begin{remark}
A different estimate, obtained from Theorem~\ref{t:main}, is found using
\begin{align}\label{eq:other}
    C_N \leq \| (\mathcal L_0 + \mathcal P^{\mathbb W}_N \mathcal L_1)^{-1} \mathcal N^{-1} (\mathcal N_0 + \mathcal N_{1,N})\mathcal L_0\|_{\mathbb V \to \mathbb V}.
\end{align}
This can be useful if something is known \emph{a priori} about $\mathcal L_0 + \mathcal P^{\mathbb W}_N \mathcal L_1$.
\end{remark}

We will also need to understand what occurs when we use a different projection operator $\mathcal I_N^{\mathbb W}$, instead of $\mathcal P_N^{\mathbb W}$.  We will suppose that 
\begin{align*}
    \|u - \mathcal I_N^{\mathbb W}u\|_{\mathbb W} \overset{N \to \infty}{\to} 0,
\end{align*}
for every $u \in \mathbb W'$ where $\mathbb W'$ is a dense subspace of $\mathbb W$.   Also, we assume that $\mathcal I_N^{\mathbb W} \mathbb W' = \ran \mathcal P_N^{\mathbb W}$.

\begin{theorem}\label{t:IN-PN}
Under the hypotheses of Theorem~\ref{t:main}, suppose, for sufficiently large $N$, that $\mathcal L_0 + \mathcal I_N^{\mathbb W} \mathcal L_1 : \ran \mathcal P_N^{\mathbb V} \to \ran \mathcal P_N^{\mathbb W}$ is invertible. Then if $f \in \mathbb W'$,
\begin{align*}
\|(\mathcal L_0 + \mathcal P_N^{\mathbb W} \mathcal L_1)^{-1}  \mathcal P_N^{\mathbb W} f &- (\mathcal L_0 + \mathcal I_N^{\mathbb W} \mathcal L_1)^{-1}  \mathcal I_N^{\mathbb W} f \|_{\mathbb V} \\
&\leq \|(\mathcal L_0 + \mathcal I_N^{\mathbb W}\mathcal L_1)^{-1}\|_{\ran \mathcal I_N^{\mathbb W} \to \ran \mathcal P_N^{\mathbb V}} \|(\mathcal I_N^{\mathbb W} - \mathcal P_N^{\mathbb W})( f  - \mathcal L_1 u_N)\|_{\mathbb W}.
\end{align*}
\end{theorem}
\begin{proof}
Consider the two equations
\begin{align*}
    ( \mathcal L_0 + \mathcal I_N^{\mathbb W} \mathcal L_1) \tilde u_N &= \mathcal I_N^{\mathbb W} f,\\
    ( \mathcal L_0 + \mathcal P_N^{\mathbb W} \mathcal L_1) u_N &= \mathcal P_N^{\mathbb W} f,\\
    ( \mathcal L_0 + \mathcal I_N^{\mathbb W} \mathcal L_1) u_N &= \mathcal (\mathcal I_N^{\mathbb W} - \mathcal P_N^{\mathbb W}) \mathcal L_1 u_N + \mathcal P_N^{\mathbb W} f.
\end{align*}
Then
\begin{align*}
    \mathcal L_0( \tilde u_N - u_N) + \mathcal I_N^{\mathbb W} \mathcal L_1 ( \tilde u_N - u_N) &= (\mathcal P_N^{\mathbb W} - \mathcal I_N^{\mathbb W})  \mathcal L_1 u_N  + (\mathcal I_N^{\mathbb W} - \mathcal P_N^{\mathbb W}) f.
\end{align*}
The claim follows.
\end{proof}

\section{Spaces of periodic functions}\label{sec:period}

We consider the standard inner product on $L^2(\mathbb T)$, where $\mathbb T = [0,2\pi)$ is the one-dimensional torus:
\begin{align*}
    \langle f, g \rangle = \frac{1}{2\pi}\int_0^{2\pi} f(\theta) \overline{g(\theta)} \dd \theta.
\end{align*}
We use $e_j(\theta) = \ee^{\ii j \theta}$, $j = 0, \pm1,\pm2,\ldots$, to denote an orthonormal basis for $L^2(\mathbb T).$  We also consider functions defined on $\mathbb U = \{ z \in \mathbb C : |z| = 1\}$.  Here the inner product is
\begin{align*}
    \langle f, g \rangle = \frac{1}{2\pi}\int_0^{2\pi} f(\ee^{\ii\theta}) \overline{g(\ee^{\ii\theta})} \dd \theta,
\end{align*}
giving rise to $L^2(\mathbb U).$

From this, we associate any $u \in L^2(\mathbb T)$ (and hence $L^2(\mathbb U)$) with a bi-infinite vector
\begin{align*}
    u \mapsto \vec u =  (\langle u, e_j \rangle)_{j=-\infty}^\infty = \begin{bmatrix} \vdots \\ u_{-2} \\ u_{-1} \\ \hline u_0 \\ u_{1} \\ \vdots \end{bmatrix}.
\end{align*}
We block the vector so that one can easily keep track of the ``zero mode" which lies just below the horizontal line.  From this we can define the standard Sobolev spaces
\begin{align*}
    H^s(\mathbb T) = \left\{ u \in L^2( \mathbb T) : c_s \sum_{j=-\infty}^\infty |u_j|^2 (\max\{1,|j|\})^{2s} < \infty \right\}, \quad s \geq 0,
\end{align*}
with the obvious norm $\|u\|_{H^s(\mathbb T)}^2 = \|u\|_s^2 = \|\vec u\|_s^2 = c_{|s|}\sum_{j=-\infty}^\infty |u_j|^2 (\max\{1,|j|\})^{2s}$.  We assume $c_{|s|}$ is chosen to so that for $s > 1/2$, $\|uv\|_s \leq \|u\|_s\|v\|_s$.  If $s < 0$, we define $H^s(\mathbb T)$ to be,
\begin{align*}
    H^s(\mathbb T) = \left\{ (u_{j})_{j=-\infty}^\infty : c_{|s|} \sum_{j=-\infty}^\infty |u_j|^2 (\max\{1,|j|\})^{2s} < \infty \right\}.
\end{align*}
The space $H^s(\mathbb U)$ is defined in precisely the same way.  Furthermore, all results established below concerning $H^s(\mathbb T)$ will extend to $H^s(\mathbb U)$ after, possibly, setting $z = \ee^{\ii \theta}$.  Lastly, we note that the norm
\begin{align*}
    \lvertiii \vec u \rvertiii_{s} = \sum_{j=-\infty}^\infty |u_j|^2 (1 + |j|)^{2s},
\end{align*}
is equivalent to the $H^s(\mathbb T)$ norm as defined above.  We use these norms interchangeably.

Then $H^s(\mathbb T)$ and $H^{-s}(\mathbb T)$ are dual with respect to the standard $\ell^2$ inner product as the pairing:
\begin{align*}
    ( u, v) = \sum_{j=-\infty}^\infty u_j v_j.
\end{align*}
Indeed, suppose $L$ is a bounded linear functional on $H^s(\mathbb T)$, then the Riesz representation theorem gives that
\begin{align*}
    L(u) = c_{|s|}\sum_{j=-\infty}^\infty  u_j w_j (\max\{1,|j|\})^{2s},
\end{align*}
for some $w \in H^s(\mathbb T)$.  But, define $v_j = w_j (\max\{1,|j|\})^{2s}$ and we have that
\begin{align*}
\|v\|_{-s}^2 = c_{|s|}\sum_{j=-\infty}^\infty |v_j|^2 (\max\{1,|j|\})^{-2s} = c_{|s|}\sum_{j=-\infty}^\infty |w_j|^2 (\max\{1,|j|\})^{2s} < \infty.
\end{align*}
Conversely, $L(u) := ( u, v)$, with $v \in H^{\mp s}(\mathbb T)$ is a bounded linear functional on $H^{\pm s}(\mathbb T),$ with
\begin{align*}
    c_{|s|}\|L\|_{\mp s} \leq \|v\|_{\pm s}.
\end{align*}
This follows from a straightforward application of the Cauchy-Schwarz inequality.  As it will be needed in what follows, for $s > 1/2$, multiplication of $u \in H^{s}(\mathbb T)$ and $v \in H^{-s}(\mathbb T)$ is a well-defined bounded linear operator on $H^{-s}(\mathbb T)$.  Indeed, for $\varphi \in H^s(\mathbb T)$ define multiplication via
\begin{align*}
    ( u v , \varphi ) := ( u, v \varphi ),
\end{align*}
and therefore
\begin{align*}
    \|u v\|_{-s} = \sup_{\varphi \in H^s(\mathbb T), \|\varphi\|_s = 1} ( u v , \varphi ) = \sup_{\varphi \in H^s(\mathbb T), \|\varphi\|_s = 1} ( u , v\varphi )  \leq \|u\|_{-s} \|v \varphi\|_s \leq C_s \|u\|_{-s} \|v\|_s,
\end{align*}
using the algebra property of $H^s(\mathbb T)$ for $s > 1/2$.

\subsection{Projections}
For $N > 0$ define $N_- = \lfloor N/2 \rfloor$ and $N_+ = \lfloor (N-1)/2 \rfloor$.  Note that $N_+ + N_- + 1 = N$.  Define the  interpolation operator\footnote{For $H^s(\mathbb U)$, one replaces $\check x_\ell$ with $\check z_\ell = \ee^{\ii \check x_\ell}$.} $\mathcal I_N$ on $H^s(\mathbb T)$, $s >1/2$ by
\begin{align*}
    \mathcal I_N u &= \sum_{j=-N_-}^{N_+} \check u_j e_j,\\
    \check u_j &= \frac{1}{N \sqrt{2 \pi}} \sum_{\ell=1}^N u(\check x_\ell) e_{-j}(\check x_\ell), \quad \check x_\ell = 2 \pi \frac{\ell -1}{N},
\end{align*}
and the Fourier truncation projection on $H^s(\mathbb T)$ for $s \in \mathbb R$
\begin{align*}
    \mathcal P_N u = \sum_{j=-N_-}^{N_+} u_j e_j.
\end{align*}
It follows from the aliasing formula that
\begin{align*}
    \check u_j = \sum_{p = - \infty}^\infty u_{p N + j}.
\end{align*}
One then has the following bounds \cite{Kress1993}.
\begin{theorem}\label{t:proj-Hs}
If $s \geq t$ then there exist a constant $D_{t,s}$ such that
\begin{align*}
    \|u - \mathcal P_N u\|_t \leq D_{t,s} N^{t -s} \|u\|_s.
\end{align*}
If, in addition $ s > 1/2$, there exist a constant $C_{t,s}$
\begin{align*}
    \|\mathcal I_N u - \mathcal P_N u\|_t \leq C_{t,s} N^{t-s} \|u\|_s.
\end{align*}
\end{theorem}
We have an immediate corollary.
\begin{corollary}
If $s \geq t$ and $s > 1/2$ then there exists a constant $E_{t,s} > 0$ such that
\begin{align*}
    \| u -\mathcal I_N u \|_t \leq E_{t,s} N^{t-s} \|u\|_s.
\end{align*}
\end{corollary}

\section{Periodic differential operators}  \label{sec:diffop}
\subsection{An operator equation}



Consider the following differential operator $\mathcal L$ for $k > p$ and $\ell \geq 1$,
\begin{align}\label{eq:diff_op}
    \mathcal L u (\theta) = \underbrace{\sum_{j=q}^k c_j \frac{\dd^k u}{\dd \theta^k}(\theta)}_{\mathcal L_0u} + \underbrace{\sum_{j=0}^{p} a_j(\theta)\frac{\dd^j u}{\dd \theta^j}(\theta)}_{\mathcal L_1u}, 
    \begin{split}
    \quad a_j &\in H^{\ell}(\mathbb T), \quad j = 0,1,\ldots,p,\\ 
    c_j &\in \mathbb C, \quad j = p+1,\ldots,k, \quad c_k \neq 0.
    \end{split}
\end{align}
Typically, one should think of $q = p+1$ so that $\mathcal L_0$ contains the highest-order derivatives that are required to have constant coefficients.  But, in some situations, it is convenient to move lower-order derivatives with constant coefficients into $\mathcal L_0$.  And here $\ell$ must be sufficiently large so that multiplication by $a_j$ is bounded in the on the image of the $p$th derivative operator.  It follows immediately that $\mathcal L: H^s(\mathbb T) \to H^{s-k} (\mathbb T),$ is bounded if 
\begin{align*}
    \ell \geq  \max\{|s-p|,p,1\}.
\end{align*}
This implies that
\begin{align}\label{eq:scond}
    \min\{0, - \ell + p\} \leq s \leq \max\{2p, p + \ell\}.
\end{align}

Less restrictive conditions on the functions $a_j$ are possible but, for simplicity, we do not explore this further.  And this choice is convenient because  $\mathcal L$ is bounded from $L^2(\mathbb T)$ to $H^{-k}(\mathbb T).$  Next, define $\mathcal N_\zeta = \mathcal N$
\begin{align*}
    \mathcal N = \left( \mathcal L_0 - \zeta \id \right)^{-1},
\end{align*}
where $\zeta$ is a point not in the spectrum of $\mathcal L_0$.  For example, if $\mathcal L_0 = \frac{\dd^k}{\dd x^k}$ one can choose
\begin{align*}
    \zeta = \begin{cases} (-1)^{k/2} & k \text{ even},\\
    1 & k \text{ odd}. \end{cases}
\end{align*}
It follows that $\mathcal N: H^s(\mathbb T) \to H^{s+k}(\mathbb T)$ is bounded. We have the following.

\begin{theorem}\label{t:diff-conv}
    Suppose that $\mathcal L$ in \eqref{eq:diff_op} is invertible on $H^s(\mathbb T)$, satisfying \eqref{eq:scond} and $f \in H^{s-k}(\mathbb T)$. Then for $N$ satisfying
    \begin{align}\label{eq:Nlarge}
        \| \mathcal N \|_{s-k \to s} D_{s-p,s-k} N^{p-k} \|\mathcal L_1\|_{s \to s- p} < \frac 1 2 \frac{1}{\| \mathcal L^{-1}\mathcal N^{-1} \|_{s} },
    \end{align}
    the operator $\mathcal L_N = \mathcal L_0 + \mathcal P_N \mathcal L_1$ is invertible and there exists $c, C > 0$ with
    \begin{align*}
    C \leq 2 \| \mathcal L^{-1}\mathcal N^{-1} \|_{s} \|\mathcal N\|_{s-k \to s} \|\mathcal L_0\|_{s \to s-k},
\end{align*}
such that the unique solution of \begin{align*}
   ( \mathcal L_0 + \mathcal P_N \mathcal L_1)u_N = \mathcal P_N f, \quad u_N \in \ran \mathcal P_N,
\end{align*}
satisfies
\begin{align}\label{eq:diff-conv}
    c \| u - \mathcal P_N u\|_{s} \leq \| u - u_N\|_s \leq C \|u - \mathcal P_N u\|_s.
    \end{align}    
\end{theorem}

\begin{proof}
We consider, as an operator on $H^s(\mathbb T)$,
\begin{align*}
    \mathcal N (\mathcal L_0 + \mathcal P_N \mathcal L_1) = \mathcal N \mathcal L_0 + \mathcal N \mathcal P_N \mathcal L_1 = \id - \mathcal J + \mathcal N \mathcal P_N \mathcal L_1.
\end{align*}
It is important that $-\mathcal J + \mathcal N \mathcal P_N \mathcal L_1$ is compact.   Note that $\mathcal J: H^{s}(\mathbb T) \to H^{s + k}(\mathbb T)$ is bounded. Then
\begin{align*}
    \| \mathcal N (\id - \mathcal P_N) \mathcal L_1 \|_{s} \leq \|\mathcal N\|_{s-k \to s} \|\id - \mathcal P_N\|_{s-p \to s-k} \|\mathcal L_1\|_{s \to s-p},
\end{align*}
and we know from Theorem~\ref{t:proj-Hs}, that $\|\id - \mathcal P_N\|_{s-p \to s-k} \leq D_{s-p,s-k} N^{p-k}$.   Given that $\mathcal L$ is invertible, we then apply Theorem~\ref{t:main}, choosing $N$ sufficiently large so that \eqref{eq:Nlarge} holds, to find that the unique solution of
\begin{align*}
   ( \mathcal L_0 + \mathcal P_N \mathcal L_1)u_N = \mathcal P_N f, \quad u_N \in \ran \mathcal P_N,
\end{align*}
satisfies \eqref{eq:diff-conv}.  Tracking the dependence on constants in Theorem~\ref{t:main} gives the result.
\end{proof}

\begin{remark}
Note that also, for example, if $s +k - p > 1/2$
\begin{align*}
    \| (\mathcal I_N- \mathcal P_N) \mathcal L_1 \|_{s + k \to s} \leq \|\mathcal I_N - \mathcal P_N\|_{s +k-p \to s} \|\mathcal L_1\|_{s+ k \to s + k-p} = O(N^{p-k}).
\end{align*}
We can then apply Theorem~\ref{t:IN-PN} to use $\mathcal I_N$ in place of $\mathcal P_N$ if $f \in H^1(\mathbb T)$.
\end{remark}

\subsection{Spectrum approximation}

As an application of the framework, we discuss the sense in which the finite number of eigenvalues of the operator
\begin{align*}
    \mathcal L_0 + \mathcal P_N \mathcal L_1: \ran \mathcal P_N \to \ran \mathcal P_N,
\end{align*}
approximate the spectrum of $\mathcal L_0 +\mathcal L_1$.  We first introduce some general considerations and results.

Consider a bounded operator $\mathcal L: \mathbb V \to \mathbb W$ where $\mathbb V,\mathbb W$ are Hilbert spaces and $\mathbb V \subset \mathbb W$ with a bounded inclusion operator.  The spectrum is defined by
\begin{align*}
    \sigma(\mathcal L) = \sigma(\mathcal L;\mathbb V, \mathbb W) = \{ z \in \mathbb C: z \id - \mathcal L : \mathbb V \to \mathbb W \text{ does not have a bounded inverse} \}.
\end{align*}
One useful quantity is also the $\epsilon$-pseudospectrum: 
\begin{align}\label{eq:pseudo}
    \sigma_\epsilon(\mathcal L) = \sigma_\epsilon(\mathcal L;\mathbb V, \mathbb W) = \overline{\{ z \in \mathbb C: \|(z \id - \mathcal L)^{-1}\|_{\mathbb W \to \mathbb V} > \epsilon^{-1} \}}.
\end{align}
For any Hilbert space $\mathbb X$, with $\mathbb V \subset \mathbb X$, $\mathbb V$ dense in $\mathbb X$, and continuously embedded, we can also consider $\sigma(\mathcal L;\mathbb X, \mathbb W)$ and $\sigma_\epsilon(\mathcal L;\mathbb X, \mathbb W)$. If $\mathbb X = \mathbb W$ we just write $\sigma(\mathcal L;\mathbb W)$ or $\sigma_\epsilon(\mathcal L;\mathbb W)$.  We use the convention that $\|(z \id - \mathcal L)^{-1}\|_{\mathbb W \to \mathbb V} = \infty$ if $z \in \sigma(\mathcal L)$.  Thus $\sigma_0(\mathcal L) = \sigma(\mathcal L)$. Then we denote the complement of the $\epsilon$-pseudospectrum by
\begin{align*}
    \rho_{\epsilon}(\mathcal L) := \sigma_{\epsilon}(\mathcal L)^c.
\end{align*}

The next proposition shows that an eigenfunction will be smoother if the coefficient functions are smoother and it applies specifically to \eqref{eq:diff_op}.
\begin{proposition}\label{p:ef-est}
For $0 \leq s \leq \ell + k$ there exists a constant $C_{s}$, independent of $\lambda$, such that if $v \in L^2(\mathbb T)$  is eigenfunction with eigenvalue $\lambda$,
\begin{align*}
    (\mathcal L_0 + \mathcal L_1) v = \lambda v, \quad \|v\|_0 = 1,
\end{align*}
then
\begin{align*}
    \|v\|_{s} \leq C_s ( |\lambda| + 2)^{\lceil \frac{s}{k-p} \rceil}.
\end{align*}
\end{proposition}
\begin{proof}
  We begin with
\begin{align*}
    \mathcal L_0 v = (\lambda \id - \mathcal L_1) v.
\end{align*}
This gives
\begin{align*}
    v = \mathcal J v  + \mathcal N(\lambda \id - \mathcal L_1) v, \quad \id - \mathcal J = \mathcal N \mathcal L_0.
\end{align*}
We have three operators to consider: There exists constants $C = C(s) > 0$ such that,
\begin{align*}
    \|\lambda\mathcal N \|_{s \to s + k} &\leq C |\lambda|,\\
    \|\mathcal J\|_{s \to s + k} &\leq C ,\\
     \|\mathcal N \mathcal L_1\|_{s \to s - p + k} & \leq C,
\end{align*}
whenever these operators are bounded.  The first two come with no restrictions on $s$.  The last requires $s$ satisfying $s - p \leq \ell$ so that multiplication by $a_j$ is bounded after taking $p$ derivatives. For $s \leq \ell +p$
\begin{align}\label{eq:ef_bound}
    \|v\|_{t} \leq \|v\|_{s - p + k} \leq [C (|\lambda| + 1) + C] \|v\|_s, \quad r \leq s + k.
\end{align}
The result follows.
\end{proof}

\begin{remark}
If $\mathcal L = \frac{\dd^k}{\dd \theta^k}$, then the eigenfunctions of $\mathcal L$ are simply $v_j(\theta) = \ee^{\ii j \theta}$ with eigenvalues $\lambda_j = (\ii j)^k$. Then for $j \neq 0$
\begin{align*}
    \|v_j\|_{kn} = |j|^{kn} = |\lambda_j|^n,
\end{align*}
which shows that the previous proposition, with $p = 0$, is nearly sharp.  But for general, non-constant coefficient operators, we do not expect it to be sharp.
\end{remark}

\begin{remark}\label{rem:proj_ef_est}
If we replace $\mathcal L_1$ with $\mathcal P_N \mathcal L_1$ or $\mathcal P_N \mathcal L_1 \mathcal P_N$ in the previous theorem, the statement remains true.
\end{remark}

\begin{corollary}
For $0 < s \leq \ell +k $, the $L^2(\mathbb T)$ eigenvalues of $\mathcal L$ coincide with the $H^s(\mathbb R)$ eigenvalues of $\mathcal L$.
\end{corollary}

\begin{theorem}\label{t:eval-gen}
\begin{enumerate}
    \item Suppose $s$ satisfies \eqref{eq:scond} and $\mathcal L$ is given in \eqref{eq:diff_op}.  Then there exists a constant $c = c(\mathcal L,s)$ such that if
    \begin{align*}
        N >  c \epsilon^{1/(p-k)},
    \end{align*}
    then no eigenvalue of $\mathcal L_0 + \mathcal P_N \mathcal L_1$, restricted to $\ran \mathcal P_N$, lies in the set $\rho_{\epsilon}(\mathcal L; H^s(\mathbb T), H^{s-k}(\mathbb T))$.
   
    \item Suppose $\lambda$ is an isolated eigenvalue of $\mathcal L$  and $\delta = \delta(N,\lambda) = O(1)$ as $N \to \infty$ is such that
    \begin{align*}
        &\max_{|z-\lambda| = \delta} \| (\id - \mathcal K(z) )^{-1}\|_0 N^{p-k} = o(1),\\
        &\delta\max_{|z-\lambda| = \delta} \| (\id - \mathcal K(z) )^{-1}\|_0 \max_{|z-\lambda| = \delta}  \| (\id - \mathcal K(z) )^{-1}\|_t N^{-t } (|\lambda| +1)^{\lceil \frac{t}{k-p} \rceil} = o(1),\\
        & \id - \mathcal K(z) = \mathcal N(\mathcal L_0  + \mathcal L_1 - z \id),
    \end{align*}
    for some $0 \leq t \leq k + \ell$. Then there exists $N_0 = N_0(\delta,\lambda)$, such that for all $N > N_0$ there exists an eigenvalue $\lambda_N$ of $\mathcal L_0 + \mathcal P_N \mathcal L_1$ satisfying
    \begin{align*}
        |\lambda - \lambda_N| \leq \delta.
    \end{align*}
\end{enumerate}
\end{theorem}
\begin{proof}
We first establish (1).  We apply $\mathcal N$ so that we consider the operators
\begin{align*}
    \id - \mathcal K_N(z) \quad \text{and} \quad \id - \mathcal K(z),
\end{align*}
where
\begin{align*}
    \mathcal K(z) &= \id - \mathcal N(\mathcal L_0 - z \id + \mathcal L_1 ),\\
\mathcal K_N(z) &= \id - \mathcal N(\mathcal L_0 - z \id + \mathcal P_N \mathcal L_1 ).
\end{align*}
It follows that if $\id - \mathcal K_N(z)$ is invertible  on $H^s(\mathbb T)$, then so is \begin{align*}
 \mathcal L_0 - z \id + \mathcal P_N\mathcal L_1 ,
\end{align*}
on $\ran \mathcal P_N$.
If
\begin{align*}
    \|(\id - \mathcal K(z))^{-1}\|_{s} \|\mathcal K(z) - \mathcal K_N(z)\|_{s} < 1,
\end{align*}
then $\id - \mathcal K_N(z)$ is invertible and $z$ is not an eigenvalue of $\mathcal L_0 + \mathcal P_N \mathcal L_1$, again restricted to $\ran \mathcal P_N$.  Consider
\begin{align*}
    \|\mathcal K(z) - \mathcal K_N(z)\|_{s} = \|\mathcal K(0) - \mathcal K_N(0)\|_{s},
\end{align*}
 Let 
 \begin{align*}
     N_0 = \left(\frac{3  \|\mathcal L_1\|_{s \to s-p} \|\mathcal N^{-1}\|_{s \to s-k}\|\mathcal N\|_{s-k \to s} D_{s-p,s-k}}{2 \epsilon}\right)^{1/(k-p)},
 \end{align*}
 where $D_{s-p,s-k}$ is the same constant as in Theorem~\ref{t:proj-Hs}.
 We then estimate, for $N > N_0$,
 \begin{align*}
     \|\mathcal K(0) - \mathcal K_N(0)\|_{s} = \| \mathcal N(\id - \mathcal P_N) \mathcal L_1\|_s &\leq   \|\mathcal N\|_{s -k  \to s} \|\id - \mathcal P_N\|_{s -p \to s-k}\|\mathcal L_1\|_{s \to s-p}\\
     &\leq  D_{s-p,s-k} N^{p-k} \|\mathcal N\|_{s-k \to s} \|\mathcal L_1\|_{s \to s-p} \leq \frac{2 \epsilon}{3\|\mathcal N^{-1}\|_{s \to s-k}}.
 \end{align*}
 Thus for $z \in \rho_{\epsilon}(\mathcal L; H^s(\mathbb T), H^{s-k}(\mathbb T))$,
\begin{align*}
    \|(\id - \mathcal K(z))^{-1}\|_{s} \|\mathcal K(z) - \mathcal K_N(z)\|_{s} &\leq \frac{2 \epsilon \|(\mathcal L- z \id)^{-1} \mathcal N^{-1}\|_{s}}{3\|\mathcal N^{-1}\|_{s \to s-k}}\\
    & \leq  \frac{2 \epsilon\|(z\id - \mathcal L)^{-1}\|_{s -k \to s}}{3} \leq \frac 2 3,
\end{align*}
if $N > N_0$.  This establishes (1).

For (2), if $\lambda$ is an eigenvalue of $\mathcal L$, $\mathcal L f = \lambda f$, $\|f\|_0 = 1$, consider
\begin{align*}
    f_N &:= \frac{1}{2 \pi \ii} \oint_{|z - \lambda| = \delta} (z\id - \mathcal L_0 - \mathcal P_N \mathcal L_1 )^{-1} \mathcal P_N f \dd z,\\
    f &= \frac{1}{2 \pi \ii} \oint_{|z - \lambda| = \delta} (z \id - \mathcal L)^{-1} f \dd z
\end{align*}
where $\delta$ is chosen sufficiently small so that the contour of integration does not encircle any other eigenvalues of $\mathcal L$.

If we show that for a given value of $\delta$, that $f_N \neq 0$, then there must be an eigenvalue of $\mathcal L_0 + \mathcal P_N \mathcal L_1$ encircled by, or on, the contour of integration.  If an eigenvalue lies on the contour, we are done, so suppose this does not occur and we just need to show that $f_N \neq 0$. The first simple estimate is
\begin{align*}
    \|f - f_N\|_0 \leq \delta  \max_{|z - \lambda| = \delta}\| (z\id - \mathcal L_0 - \mathcal P_N \mathcal L_1 )^{-1} \mathcal P_Nf - (z \id - \mathcal L)^{-1} f \|_0.
\end{align*}
Replacing $f$ with $-f$, we must consider the operator equation
\begin{align*}
    (\mathcal L - z \id )u  = f,
\end{align*}
and it is approximation
\begin{align*}
    \left( \mathcal L_0 + \mathcal P_N \mathcal L_1 - z \id \right)u_N = \mathcal P_N f, \quad u_N \in \ran \mathcal P_N.
\end{align*}
Using the above estimates and notation, with $s = 0$,
\begin{align*}
    &\| (\id - \mathcal K(z) )^{-1}\|_0 \|\mathcal K(z) - \mathcal K_N(z)\|_0 \leq H_{p,k}  \gamma(\delta) N^{p-k},\\
    &H_{p,k}: = D_{-p,-k} \|\mathcal N\|_{-p \to k-p} \|\mathcal L_1\|_{0 \to -p},\\
    &\gamma(\delta) := \max_{|z-\lambda| = \delta} \| (\id - \mathcal K(z) )^{-1}\|_0
\end{align*}
Then Theorem~\ref{t:main} applies, for sufficiently large $N$, $H_{p,k}\gamma(\delta) N^{p-k} < 1$,
\begin{align*}
    \|f - f_N\|_0 &\leq  \delta \frac{\gamma(\delta)}{1 - H_{p,k}\gamma(\delta) N^{p-k}}\|\mathcal N\|_{-k \to 0} \|\mathcal L_0\|_{0 \to -k} \max_{|z - \lambda| = \delta}\|u(z) - \mathcal P_N u(z)\|_0,\\
    u(z) &= (\mathcal L - z \id )^{-1} f.
\end{align*}
Then, we estimate, using Proposition~\ref{p:ef-est},
\begin{align*}
    \|u(z) - \mathcal P_N u(z)\|_0 &\leq \|( \id - \mathcal P_N) (\id - \mathcal K(z))^{-1} f\|_0 \leq D_{0,t} N^{-t} \|(\id - \mathcal K(z))^{-1} f\|_t \\
    &\leq D_{0,t} N^{-t} \max_{|z-\lambda| = \delta} \|(\id - \mathcal K(z))^{-1}\|_{t}\|f\|_t.
\end{align*}
Then $\|f\|_t$ can be estimated using Proposition~\ref{p:ef-est}: For $0 \leq t \leq \ell + k$
\begin{align*}
    \|f\|_t \leq C_t (|\lambda| +1)^{\lceil \frac{t}{k-p} \rceil}.
\end{align*}
Therefore
\begin{align*}
    \|f - f_N\|_0 \leq  \delta \frac{\gamma(\delta)}{1 - H_{p,k}\gamma(\delta) N^{p-k}}\|\mathcal N\|_{-k \to 0} \|\mathcal L_0\|_{0 \to -k} D_{0,t} N^{-t} \max_{|z-\lambda| = \delta} \|(\id - \mathcal K(z))^{-1}\|_{t}\|f\|_t.
\end{align*}
\end{proof}

While Theorem~\ref{t:eval-gen} is general, it does not provide much in terms of convergence rates.  It states that all eigenvalues eventually enter the $\epsilon$-pseudospectrum.  The second part of the theorem applies, in particular, for each $\delta$ and $\lambda$ fixed, establishing that every true eigenvalue is approximated.  If more is known about the operator under consideration, rates become apparent.  For example, if we consider self-adjoint operators, we can obtain much stronger results.

\subsection{Improvements for self-adjoint operators}\label{sec:sa}

The first results of this section apply much more generally.  In this section the resolvent operators, i.e., operators of the form $(z \id - \mathcal L)^{-1}$, play the role of the preconditioning operator $\mathcal N$.  And the forthcoming results could be established for normal operators, but we do not pursue that here.
\begin{theorem}\label{t:sa-mult}
Consider bounded self-adjoint operators\footnote{By self-adjoint, we mean that operator in question is self-adjoint on $\mathbb W$ with dense domain $\mathbb V$.} $\mathcal L_N, \mathcal L: \mathbb V \to \mathbb W$ on Hilbert spaces $\mathbb V, \mathbb W$, where $\mathbb V \subset \mathbb W$ and the inclusion operator is bounded.  Suppose:
\begin{enumerate}
    \item $\mathcal L_N, \mathcal L$ have pure-point spectrum.
    \item $\mathbb V$ is dense in $\mathbb W$ and $\mathbb V$ is separable.
    \item  $\lambda$ is an eigenvalue of $\mathcal L$ that is of finite geometric multiplicity and $\delta> 0$ is such that $\lambda$ is a distance more than $3\delta$ away from all other elements of the spectrum of $\mathcal L$.
    \item There exists $N_0 = N_0(\lambda,\delta)$ such that for $\delta  \leq |z-\lambda| \leq  2\delta$, $z\id - \mathcal L_N$ is invertible on $\mathbb W$ for $N > N_0$ and 
\begin{align*}
\|(z \id - \mathcal L)^{-1} (\mathcal L - \mathcal L_N) v\|_{\mathbb W} \overset{N \to \infty}{\to} 0 ,\\
\|(z \id - \mathcal L_N)^{-1} (\mathcal L - \mathcal L_N) u\|_{\mathbb W} \overset{N \to \infty}{\to} 0,
\end{align*}
uniformly in $z$, $\delta  \leq |z-\lambda| \leq  2\delta$, for all $\mathbb W$-normalized eigenfunctions $v$ of $\mathcal L_N$ and $u$ of $\mathcal L$ associated to eigenvalues that lie in  $\{z \in \mathbb C: |z - \lambda| \leq \delta\}$.
\end{enumerate}
Then there exists $N_0'$ such that for $N > N_0'$ the geometric multiplicity of $\lambda$ is equal to the sum of the geometric multiplicities of the eigenvalues $\lambda'$ of $\mathcal L_N$ satisfying $|\lambda' - \lambda| < \delta$.
\end{theorem}

\begin{proof}
Let $m$ be the geometric multiplicity of $\lambda$ and let $n = n(N,\delta)$ be the sum of the geometric multiplicities of the eigenvalues $\lambda'$ of $\mathcal L_N$ satisfying $|\lambda' - \lambda| < \delta$. 
The spectral theorem gives
\begin{align*}
    \mathcal L u &= \sum_{j=1}^\infty \lambda_j \langle u, \varphi_j \rangle_{\mathbb W} \varphi_j,\\
    \mathcal L_N u &= \sum_{j=1}^\infty \lambda_j^{(N)} \langle u, \varphi_j^{(N)} \rangle_{\mathbb W} \varphi_j^{(N)}.
\end{align*}
Let $v_1,\ldots, v_m$ be an orthonormal basis for the eigenspace associated to $\lambda$.  Then consider the integrals
\begin{align*}
    v_{N,j}  &= \frac{1}{2 \pi\ii} \oint_{|z-\lambda| = \delta} (z \id - \mathcal L_N)^{-1} v_j \dd z,\\
     v_{j} &= \frac{1}{2 \pi\ii} \oint_{|z-\lambda| = \delta} (z \id - \mathcal L)^{-1} v_j \dd z, \quad j = 1,\ldots,m.
    \end{align*}    
    We write $(z \id - \mathcal L_N) - ( z \id - \mathcal L)  = (\mathcal L - \mathcal L_N)$, giving, for $N > N_0$ the second resolvent identity
    \begin{align*}
        (z \id - \mathcal L)^{-1} - (z \id - \mathcal L_N)^{-1} = (z \id - \mathcal L_N)^{-1} (\mathcal L - \mathcal L_N) (z \id - \mathcal L)^{-1}.
    \end{align*}
    This gives
    \begin{align*}
        (z \id - \mathcal L)^{-1} v_j - (z \id - \mathcal L_N)^{-1} v_j = (z \id - \mathcal L_N)^{-1} (\mathcal L - \mathcal L_N) \frac{v_j}{z - \lambda}.
    \end{align*}
    Since $|z - \lambda| = \delta$, under the stated assumptions, $\|v_j - v_{N,j}\|_{\mathbb W} \to 0$ as $N\to \infty$.  This shows that for sufficiently large $N$, $N > N_0'$
    \begin{align*}
        n(N,\delta) = \dim \ran \frac{1}{2 \pi\ii} \oint_{|z-\lambda| = \delta} (z \id - \mathcal L_N)^{-1} \dd z \geq m.
    \end{align*}
    Next, let $v_1^{(N)},\ldots, v_n^{(N)}$ be an orthonormal basis for the span of the eigenspaces for eigenvalues $\lambda_j^{(N)}$ of $\mathcal L_N$ satisfying $|\lambda - \lambda_j^{(N)}| < \delta$.  And then consider the integrals, for $N > N_0$
    \begin{align*}
        w_{N,j} &= \frac{1}{2 \pi\ii} \oint_{|z-\lambda| = 2\delta} (z \id - \mathcal L)^{-1} v_j^{(N)} \dd z.
    \end{align*}
    Then it follows that
    \begin{align*}
        (z \id - \mathcal L)^{-1} v_j^{(N)} - (z \id - \mathcal L_N)^{-1} v_j^{(N)} = (z \id - \mathcal L)^{-1} (\mathcal L_N - \mathcal L) \frac{v_j^{(N)}}{z - \lambda_j^{(N)}}. 
    \end{align*}
    Since $|z - \lambda_j^{(N)}| \geq \delta$, under the stated assumptions, $\|v_j^{(N)} - w_{N,j}\|_{\mathbb W} \to 0$ as $N\to \infty$.  This shows that for sufficiently large $N$, $N > N_0'$ (increasing $N_0'$ if necessary),
    \begin{align*}
        m = \dim \ran \frac{1}{2 \pi\ii} \oint_{|z-\lambda| = 2\delta} (z \id - \mathcal L)^{-1} \dd z \geq n(N,\delta).
    \end{align*}    
    This proves the claim.
    
\end{proof}


With the previous theorem in hand, we can now look to estimate the rate of convergence.  The following will apply when $\mathcal L_N$ is normal, and, in particular, self-adjoint.  

\begin{theorem}\label{t:sa-rate}
Suppose the assumptions of Theorem~\ref{t:sa-mult}. 
Assume further that there exists constants $t, N_0 > 0$ and a continuous function $R: [0,\infty) \to [0,\infty)$ such that
\begin{align*}
    \|(z \id - \mathcal L_N)^{-1} (\mathcal L - \mathcal L_N) u\|_{\mathbb W} &\leq R(|\lambda|) \max_{\mu \in \sigma(\mathcal L_N)} \frac{1}{|z - \mu|} N^{-t},
\end{align*}
for $N > N_0$ whenever $u$ is a normalized eigenfunction of $\mathcal L$ associated to an eigenvalue with modulus at most $|\lambda|$.  Then if $\lambda_N$ is an eigenvalue of $\mathcal L_N$ such that $|\lambda_N - \lambda| < \delta$ for all $N$ sufficiently large then there exists $C, N_0' > 1$ such that
\begin{align}
    |\lambda - \lambda_N| \leq  C R(|\lambda|) N^{-t}, 
\end{align}
for $N > N_0'$.
\end{theorem}
\begin{proof}
In this proof we reuse the notation of Theorem~\ref{t:sa-mult}.  
Consider the rescaled points
\begin{align*}
    z_j^{(N)} = \frac{\lambda_j^{(N)} - \lambda}{\max_j |\lambda_j^{(N)} - \lambda|},
\end{align*}
for values of $N$ such that the denominator does not vanish.  The maximum is taken over $j$ such that $|\lambda_j(N) - \lambda| < \delta$.  If the denominator does vanish, our estimate will hold trivially.  For sufficiently large $N$, by Theorem~\ref{t:sa-mult}, there are at most $m$ of points.  Thus there exists radii $r_N >0$ satisfying
\begin{align*}
    \gamma \leq r_N \leq 1 - \gamma, \quad | r_N - |z_j^{(N)}|| > \gamma,
\end{align*}
for some fixed $0 < \gamma < 1/2$.  Set 
\begin{align*}
    \delta' = r_N \max_j |\lambda- \lambda_j^{(N)}|.
\end{align*}
This is chosen so that any $z$ satisfying $|z- \lambda| = \delta'$ is a distance at least $\gamma \max_j |\lambda- \lambda_j^{(N)}|$ from $\sigma(\mathcal L_N)$.

Note that at least one of the $\lambda_j^{(N)}$'s is outside the disk $\{ z \in \mathbb C : |z- \lambda| \leq \delta '\}$.  
We then estimate
\begin{align*}
    \| v_{N,j} - v_j\|_{\mathbb W} &\leq \frac{\delta'}{2 \pi } \max_{|z - \lambda| = \delta'} \frac{\|(z \id - \mathcal L_N)^{-1} (\mathcal L - \mathcal L_N) v_j\|_W}{|z - \lambda|}\\
    & \leq R(|\lambda|) N^{-t} \max_{|z - \lambda| = \delta'}\max_{\mu \in \sigma(\mathcal L_N)}\frac{1}{|z - \mu|}\\
    & \leq  \frac{R(|\lambda|) N^{-t}}{\gamma \max_j |\lambda- \lambda_j^{(N)}|}.
\end{align*}
By the orthogonality of the eigenfunctions,  there exists a constant $c_m > 0$ such that if
\begin{align*}
     \| v_{N,j} - v_j\|_{\mathbb W} &\leq c_m, \quad j = 1,2,\ldots,m,
\end{align*}
then the rank of the projector
\begin{align*}
    \frac{1}{2 \pi \ii} \oint_{|z-\lambda| = \delta'} (z \id - \mathcal L_N)^{-1} \dd z,
\end{align*}
is at least $m$.   We know that, for sufficiently large $N$, the rank must be less than $m$, and therefore
\begin{align*}
    \frac{R(|\lambda|) N^{-t}}{\gamma \max_j |\lambda- \lambda_j^{(N)}|} \geq c_m, \quad N > N_0'.
\end{align*}
Therefore
\begin{align*}
    \max_j |\lambda- \lambda_j^{(N)}| \leq \frac{R(|\lambda|) N^{-t}}{\gamma c_m}, \quad N > N_0'.
\end{align*}
\end{proof}

We now apply the last few results to understand the eigenvalues of $\mathcal L_N' := \mathcal L_0 + \mathcal P_N \mathcal L_1$ restricted to $\ran \mathcal P_N$. We cannot apply them directly because $\mathcal L_N'$ is not self-adjoint on the entire space.  Instead, we consider (with some abuse of notation) $\mathcal L_N := \mathcal L_0 + \mathcal P_N \mathcal L_1 \mathcal P_N$, which is self-adjoint if both $\mathcal L_0, \mathcal L_1$ are.

\begin{theorem}\label{t:spectrum-main}
Suppose $\mathcal L$ is given by \eqref{eq:diff_op}.  Suppose further that $\lambda$ is an eigenvalue of $\mathcal L$ and $\delta > 0$ is such that $\lambda$ is a distance at least $3\delta$ from all other eigenvalues of $\mathcal L$.  There exists $N_0(\delta,\lambda)$ such that for $N > N_0$:
\begin{itemize} 
\item The sum of the geometric multiplicities of the eigenvalues of $\mathcal L_N$ that lie within in the disk $\{z \in \mathbb C: |z - \lambda| < \delta\}$ is equal to the geometric multiplicity of $\lambda$.
\item If $\lambda_N$ is an eigenvalue of $\mathcal L_N$ satisfying $|\lambda - \lambda_N| < \delta$ then
\begin{align*}
|\lambda - \lambda_N| \leq C N^{-\ell},
\end{align*}
for a constant $C > 0$.
\end{itemize}
\end{theorem}

\begin{proof}

First consider
\begin{align}\label{eq:begin}
    \|( z \id - \mathcal L)^{-1} (\mathcal L_1 - \mathcal P_N \mathcal L_1 \mathcal P_N) v\|_0
\end{align}
where $v$ is an eigenfunction of $\mathcal L_N$.  Note that if $u \in \ran \mathcal P_N$ we can replace $\mathcal P_N \mathcal L_1 \mathcal P_N$ with $\mathcal P_N \mathcal L_1$.  To this end, suppose $\mu$ is an eigenvalue of $\mathcal L_N$:
\begin{align*}
    (\mu \id - \mathcal L_0 - \mathcal P_N \mathcal L_1\mathcal P_N) v = 0, \quad v \neq 0.
\end{align*}
Decompose $v = v_1 + v_2$, $v_1 = \mathcal P_N v$, $v_2 = (\id - \mathcal P_N) v$.  Then
\begin{align}\label{eq:orth=}
    -(\mu \id - \mathcal L_0) v_2  = (\mu \id - \mathcal L_0 - \mathcal P_N \mathcal L_1\mathcal P_N) v_1.  
\end{align}
The left and right-hand sides are orthogonal, and therefore they both must be zero.  Furthermore if $|\mu| = o(N^{k})$, then $\mu \id - \mathcal L_0$ will be invertible on $\ran (\id - \mathcal P_N)$, for sufficiently large $N$, implying that $v_2 = 0$.  Thus, in this regime, for $z \not\in \sigma(\mathcal L)$ and $t > p$,
\begin{align*}
    \|( z \id - \mathcal L)^{-1} (\mathcal L_1 - \mathcal P_N \mathcal L_1 \mathcal P_N) v\|_0 &= \|( z \id - \mathcal L)^{-1} (\mathcal L_1 - \mathcal P_N \mathcal L_1 ) v\|_0 \\
    &\leq \|(z \id - \mathcal L)^{-1}\|_0 \|\id -\mathcal P_N\|_{t \to 0} \|\mathcal L_1\|_{t +p \to t} \|v\|_{t+p}.
\end{align*}
This final norm is finite, by Proposition~\ref{p:ef-est}, if $t = \ell$:
\begin{align*}
    \|v\|_{\ell+p} \leq C_{t + p} ( 2 + |\lambda|)^{\lceil \frac{ \ell}{k -p} \rceil}.
\end{align*}

Combining the bounds,
\begin{align*}
    \|( z \id - \mathcal L)^{-1} (\mathcal L_1 - \mathcal P_N \mathcal L_1 \mathcal P_N) v\|_0 \leq \|(z \id - \mathcal L)^{-1}\|_0 \underbrace{D_{t,0} \|\mathcal L_1\|_{t +p \to t} C_{t + p} ( 2 + |\lambda|)^{\lceil \frac{ \ell}{k -p} \rceil}}_{R_1(|\lambda|)} N^{-t}.
\end{align*}
This can then be bounded uniformly in $z$, $|z - \lambda| = \delta$, $\lambda \in \sigma(\mathcal L)$, using  $\|(z \id - \mathcal L)^{-1}\|_0 = \delta^{-1}$ provided that $\lambda$ is a distance at least $2 \delta$ from all other elements of $\sigma(\mathcal L)$.

Next, for a normalized eigenfunction $u$ of $\mathcal L$, we estimate
\begin{align*}
    \|( z \id - \mathcal L_N)^{-1} (\mathcal L_1 - \mathcal P_N \mathcal L_1 \mathcal P_N) u\|_0 \leq \|( z \id - \mathcal L_N)^{-1}\|_0 \|(\mathcal L_1 - \mathcal P_N \mathcal L_1 \mathcal P_N) u\|_0.
\end{align*}
By self-adjointness, we have
\begin{align*}
    \|( z \id - \mathcal L_N)^{-1}\|_0 = \max_{\mu \in \sigma(\mathcal L_N)} \frac{1}{|z - \mu|}.
\end{align*}
Then
\begin{align*}
    \|(\mathcal L_1 - \mathcal P_N \mathcal L_1 \mathcal P_N) u\|_0 \leq \|(\mathcal L_1 - \mathcal P_N \mathcal L_1) u\|_0 + \|(\mathcal P_N\mathcal L_1 - \mathcal P_N \mathcal L_1 \mathcal P_N) u\|_0.
\end{align*}
The first term can be estimated as above.  For the second term,
\begin{align*}
    \|(\mathcal P_N\mathcal L_1 - \mathcal P_N \mathcal L_1 \mathcal P_N) u\|_0 &\leq \|\mathcal L_1 (\id - \mathcal P_N) u\|_0 \leq \|\mathcal L_1\|_{p \to 0} \|(\id - \mathcal P_N) u\|_p\\
    & \leq \|\mathcal L_1\|_{p \to 0} D_{t +p, p} N^{-t} \|u\|_{p+t}\\
    & \leq \underbrace{\|\mathcal L_1\|_{p \to 0} D_{t +p, p}C_{t + p} ( 2 + |\lambda|)^{\lceil \frac{ \ell}{k -p} \rceil}}_{R_2(|\lambda|)}  N^{-t}.
\end{align*}

The only assumption of Theorem~\ref{t:sa-mult} that remains to be established is that $z \id - \mathcal L_N$ is invertible for $\delta \leq |z - \lambda| \leq 2 \delta$ for $N$ sufficiently large, independently of $z$.  To this end, we apply Theorem~\ref{t:diff-conv} with $s = k$ and  $\mathcal L$ replaced with $\mathcal L - z \id$.  Since $\mathcal N$ is independent of $z$, we obtain the condition on $N$ 
\begin{align*}
    \frac 1 2 D_{k-p,k-k}^{-1} N^{k -p} > \|\mathcal N\|_{0\to k}\|\mathcal (z \id - \mathcal L)^{-1}\|_{0 \to k} \|\mathcal N^{-1}\|_{k \to 0}  \|\mathcal L_1\|_{k\to k-p}
\end{align*}
Then for $\delta \leq |z - \lambda| \leq 2 \delta$,
by the Open Mapping Theorem, $(z \id - \mathcal L)^{-1}$ is bounded from $L^2(\mathbb T) \to H^k(\mathbb T)$. Then
\begin{align*}
z \mapsto \|(z \id - \mathcal L)^{-1}\|_{0\to k},
\end{align*}
is a bounded, continuous function for $\delta \leq |z - \lambda| \leq 2 \delta$ and is therefore uniformly bounded. 

This implies that Theorems~\ref{t:sa-mult} and \ref{t:sa-rate} hold for $\mathcal L$, $\mathcal L_N$, and any fixed $\delta$ sufficiently small. Further, $R(|\lambda|)$ in Theorem~\ref{t:sa-rate} can be given by $R(|\lambda|) = 2R_1(|\lambda|) + R_2(|\lambda|)$.  
\end{proof}

\begin{remark}
    In the proof of the previous theorem, if an explicit $\lambda$-dependent bound on $$\|(z \id - \mathcal L)^{-1}\|_{0 \to k},$$ could be established, then a choice for $N_0$ could be made fully explicitly.  Furthermore, we expect the constant $C$ in this theorem to grow no worse than
    \begin{align*}
    |\lambda|^{\lceil \frac{\ell}{k -p}\rceil},
    \end{align*}    
    but this is likely quite pessimistic.
\end{remark}

On gap remains:  What do the eigenvalues of $\mathcal L_0 + \mathcal P_N \mathcal L_1$ restricted to $\ran \mathcal P_N$ have to do with those of $\mathcal L_N$? Indeed, provided that the eigenvalue under consideration is not too large, measured relative to $N$, the eigenvalues coincide.  

\begin{theorem}
    For any $c>0$,
\begin{align*}
    \sigma(\mathcal L_0 + \mathcal P_N \mathcal L_1; \ran P_N) \cap \{ z \in \mathbb C : |z| \leq c N^{k-1}\} = \sigma(\mathcal L_N; L^2(\mathbb T)) \cap \{ z \in \mathbb C : |z| \leq c N^{k-1}\},
\end{align*}
for sufficiently large $N$.
\end{theorem}
\begin{proof}
    Indeed, if $\lambda = O(N^{k-1})$ then following the argument that led to $v_2 = 0$ in \eqref{eq:orth=} we establish the result.
\end{proof}

\begin{remark}
    Our numerical experiments demonstrate that the exponent of $\ell$ in Theorem~\ref{t:spectrum-main} is not optimal.  The true rate appears to be closer to $2 \ell + k$. To obtain better rates, one coulde replace the $L^2(\mathbb T) = H^0(\mathbb T)$ norm in \eqref{eq:begin} with the weakest possible norm for which $(z \id - \mathcal L)^{-1}$ is bounded.  If this can be taken to be $H^{-\ell - k}(\mathbb T)$ then something similar to $2 \ell + k$ will arise.  This is an interesting question we leave for future work.
\end{remark}

\section{Riemann--Hilbert problems on the circle}\label{sec:rhp}

We consider the problem of finding a sectionally analytic function $\phi: \mathbb C \setminus \mathbb U \to \mathbb C$ that satisfies the jump condition
\begin{align}\label{eq:rhp}
    \phi^+(z) =  \phi^-(z) g(z), \quad z \in \mathbb U,  \quad \phi^\pm(z) = \lim_{\epsilon \to 0^+} \phi(z (1 \mp \epsilon)),
\end{align}
along with the asymptotic condition
\begin{align*}
    \lim_{z \to \infty} \phi(z) = 1.
\end{align*}
This is called a Riemann--Hilbert (RH) problem. In principle, one has to be precise about the sense in which these limits exist.  For simplicity, we require solutions to be of the form
\begin{align} \label{eq:solnform}
    \phi(z) = 1 + \frac{1}{2 \pi \ii} \oint_{\mathbb U} \frac{u(z')}{z' - z}\dd z' := 1 + \mathcal C u(z)
\end{align}
for $u \in L^2(\mathbb U)$.  We then define the boundary values
\begin{align*}
    \mathcal C^\pm u(z) = \lim_{\epsilon \to 0^+} \mathcal C u(z (1 \mp \epsilon)),
\end{align*}
which exist for a.e. $z \in \mathbb U$ \cite{Duren}.  For $u \in L^2(\mathbb U)$,
\begin{align*}
    u(z) = \sum_{j=-\infty}^\infty u_j z^j,
\end{align*}
Then for a.e. $z \in \mathbb U$,
\begin{align*}
    \mathcal C^+ u(z) = \sum_{j = 1}^\infty u_j z^j, \quad \mathcal C^- u(z) = -\sum_{j=-\infty}^{-1} u_j z^j,
\end{align*}
and therefore $\|\mathcal C^\pm\|_s = 1$ for all $s$.   Thus, it follows that a solution of the RH problem of the form \eqref{eq:solnform} exists if and only if the singular integral equations (SIE)
\begin{align}\label{eq:SIE}
    \mathcal C^+ u -  (\mathcal C^- u) g = g - 1,
\end{align}
has a solution.  From here, we will assume, at a minimum, that $g \in H^1(\mathbb U)$, so that the singular integral operator
\begin{align*}
    \mathcal C[g] u := \mathcal C^+ u -  (\mathcal C^- u) g,
\end{align*}
is bounded on $L^2(\mathbb U)$. It will be convenient at times in what follows to think of multiplication by a function as a bi-infinite matrix so we define the multiplication operator $\mathcal M(g)$ by
\begin{align*}
    \mathcal M(g) u  := g u.
\end{align*}

Supposing that $|g| > 0$ on $\mathbb U$, it follows that the operator $\mathcal C[g^{-1}]$ is a Fredholm regulator for $\mathcal C[g]$.  Indeed,
\begin{align*}
    \mathcal S[g^{-1},g] u &= \mathcal C[g^{-1}]\mathcal C[g]u  = \mathcal C[g^{-1}] (\mathcal C^+ u -  (\mathcal C^- u) g)\\
    & = \mathcal C^+ (\mathcal C^+ u -  (\mathcal C^- u) g) - (\mathcal C^-(\mathcal C^+ u -  (\mathcal C^- u) g)) g^{-1}.
\end{align*}
We then use the facts that $\mathcal C^\pm\mathcal C^\pm = \pm \mathcal C^\pm$, and $\mathcal C^+ - \mathcal C^- = \id$ to find
\begin{align*}
    \mathcal S[g^{-1},g] u & = \mathcal C^+ u -  \mathcal C^+ ((\mathcal C^- u) g) + \mathcal C^-( (\mathcal C^- u) g) g^{-1}\\
    & =  u -  \mathcal C^+ ((\mathcal C^- u) g) + \mathcal C^+( (\mathcal C^- u) g) g^{-1} \\
    & = u + \mathcal C^+( (\mathcal C^- u) g ) ( g^{-1} -1).
\end{align*}

We arrive at an important lemma. Its proof highlights a calculation that will be important in what follows.
\begin{lemma}\label{l:compact}
Suppose $h \in H^t(\mathbb U)$, $t > 1/2$. Then the operator
\begin{align*}
    \mathcal K^\pm(h):  u \mapsto \mathcal C^+ ( (\mathcal C^- u ) h) ,
\end{align*}
is compact on $H^s(\mathbb U)$, for $0 \leq s \leq t$.  Furthermore, if $0 \leq s \leq t$  then there exists a constant $B_{s,t} > 0$ such that if $s < 1/2$, and $h \in H^{t - s + 1}(\mathbb U)$ then
\begin{align*}
    \|\mathcal K^\pm(h)\|_{s \to  t} \leq B_{s,t} \|h\|_{t -s+1}
\end{align*}
and if $s \geq 1/2$, and $h \in H^{t}$ then
\begin{align*}
    \|\mathcal K^\pm(h)\|_{s \to  t} \leq B_{s,t} \|h\|_{t}.
\end{align*}
\end{lemma}
\begin{proof}
The matrix representation of the operator $u \mapsto hu$ in the standard basis is given by 
\begin{align*}
   \mathcal M(h) \vec u  := \left[ \begin{array}{ccc|cccccc}
    \ddots & \ddots & \ddots & & & \iddots \\
    \ddots &  h_0 & h_{-1} & h_{-2} \\
    \ddots & h_1 & h_0 & h_{-1} & h_{-2} \\ \hline
    &h_2 & h_1 & h_0 & h_{-1} & \ddots \\
    &&h_2 & h_1 & h_0 & \ddots \\
   \iddots &&& \ddots & \ddots & \ddots
    \end{array} \right] \begin{bmatrix} \vdots \\ u_{-2} \\ u_{-1} \\ \hline u_0 \\ u_{1} \\ \vdots \end{bmatrix}.
\end{align*}
This matrix is blocked so that when the row just below the horizontal line multiplies the vector it gives the coefficient of $z^0$.  Since $\vec h$ gives the coefficients of $h$ in its Laurent expansion on the circle we identify the two operators $\mathcal M(h) = \mathcal M(\vec h)$.  In matrix form, the operator $\mathcal C^+ ( (\mathcal C^- u) h)$ is then given by
\begin{align*}
      \mathcal K^\pm(h) :=
    &  -\left[ \begin{array}{cccccc|ccc} 
    &&&&\ddots &&& \iddots\\
    &&&& & 0 & 0\\
    \hline
     & h_j & \cdots & \cdots &h_2&h_1 & 0 \\
   \iddots & & h_j & \cdots &h_3&h_2 & & \ddots\\
    &\iddots & & \ddots & \vdots & \vdots && \\
    & &\iddots & & h_j & \vdots\\
    &&&\iddots&& h_j
   \end{array} \right].
\end{align*}
From this, it is clear that this is a finite rank operator if $h_j = 0$ for $j > M$.  This implies that the operator is a limit of finite-rank operators and is therefore compact.  Then, for
\begin{align*}
    \vec c & = \mathcal K^\pm(h) \vec u,\\
    c_k & = \sum_{j = 1}^\infty h_{j+k} u_{-j}.
\end{align*}
Then,
\begin{align*}
    |c_k|^2 &\leq   \|\vec u\|_s^2 \sum_{j=1}^\infty |h_{j+k}|^2 j^{-2s}\\
    & = \|\vec u\|_s^2 \sum_{j=1}^\infty (j + k)^{2t} |h_{j+k}|^2 j^{-2s} (j+k)^{-2t}\\
    & = \|\vec u\|_s^2 \sum_{j=1+k}^\infty j^{2t} |h_{j}|^2 (j-k)^{-2s} j^{-2t}
\end{align*}
From here, we find
\begin{align*}
    \sum_{k = 1}^\infty k^{2(s + t)}|c_k|^2 &\leq \|\vec u\|_s^2 \sum_{k = 1}^\infty \sum_{j=1+k}^\infty j^{2t} |h_{j}|^2 (j-k)^{-2s} j^{-2t} k^{2( s + t)}\\
    & = \|\vec u\|_s^2 \sum_{j=2}^\infty \sum_{k = 1}^{j-1}  j^{2t} |h_{j}|^2 (j-k)^{-2s} j^{-2t} k^{2( s + t)}.
\end{align*}
The $k = 0$ term can be estimated separately. We then must estimate
\begin{align*}
    j^{-2t} \sum_{k = 1}^{j-1} k^{2( s + t)}(j-k)^{-2s}  = j \left( \frac{1}{j}\sum_{k = 1}^{j-1} \left( \frac{k}{j}\right)^{2( s + t)} \left(1 - \frac{k}{j}\right)^{-2s} \right)
\end{align*}
The quantity in parentheses is bounded by $\int_0^{1} x^{2(s+t)} (1-x)^{-2s} \dd s$.  If $s \geq 1/2$, this bound becomes useless and we replace it with 
\begin{align*}
    \frac{1}{j}\sum_{k = 1}^{j-1} \left( \frac{k}{j}\right)^{2( s + t)} \left(1 - \frac{k}{j}\right)^{-2s}  \leq \int_0^{1 - 1/j} x^{2(s+t)} (1-x)^{-2s} \dd s + 1/j (1-1/j)^{2(s+t)} j^{2s} \leq B_{s,t} j^{2 s -1}
\end{align*}
Thus, we find that if $s < 1/2$ and $h \in H^{t + 1}(\mathbb U)$ then
\begin{align*}
    \|\mathcal K^\pm(h)\|_{s \to  s + t} < \infty,
\end{align*}
and if $s \geq 1/2$, and $h \in H^{s + t}$ then
\begin{align*}
    \|\mathcal K^\pm(h)\|_{s \to  s + t} < \infty.
\end{align*}
Accounting for the norm of $h$ completes the proof.

\end{proof}

The proof of the previous lemma indicates that the structure of the matrices that come out of these calculations is important.  For this reason, we pivot to work exclusively with bi-infinite vectors instead of functions defined on $\mathbb U$.  While, of course, equivalent, the matrix representations are more convenient.  With the notation in the proof of the previous lemma, we write \eqref{eq:SIE} 
\begin{align*}
    (\mathcal C^+ - \mathcal M(g-1) \mathcal C^- ) \vec u = (\underbrace{\id}_{\mathcal L_0} - \underbrace{\mathcal M(g-1) \mathcal C^-}_{- \mathcal L_1}) \vec u = \vec g - \vec 1.
\end{align*}
Since, of course, $\mathcal P_N$ commutes with $\id$, this operator is admissible with
\begin{align*}
    \mathcal N_0 = \id,\quad \mathcal N_1 = - \mathcal M(g^{-1}-1) \mathcal C^- = - \mathcal M(g^{-1}) \mathcal C^- + \mathcal C^-.
\end{align*}
We then consider the approximations
\begin{align}\label{eq:approx_SIE_P}
    (\id - \mathcal P_N \mathcal M(g-1) \mathcal C^-) \vec u_N  = \mathcal P_N (\vec g - \vec 1), \quad \vec u_N \in \ran \mathcal P_N, \\
     (\id - \mathcal I_N \mathcal M(g-1) \mathcal C^-) \vec u_N =  \mathcal I_N (\vec g - \vec 1), \quad \vec u_N \in \ran \mathcal P_N.\label{eq:approx_SIE_I}
\end{align}
We pause briefly to highlight an important guiding principle.
\begin{lemma}\label{l:truncation}
Suppose $s > 1/2$ and $g \in H^s(\mathbb U)$.  Then there exists a constant $C_s > 0$
\begin{align*}
    \|\mathcal I_N \mathcal M(g) \mathcal C^+ - \mathcal I_N \mathcal M(\mathcal P_M g) \mathcal C^+\|_s \leq C_s\|g - \mathcal P_Mg\|_s,\\
    \|\mathcal P_N \mathcal M(g) \mathcal C^+ - \mathcal P_N \mathcal M(\mathcal P_M g) \mathcal C^+\|_s \leq C_s\|g - \mathcal P_Mg\|_s.
\end{align*}
\end{lemma}
\begin{proof}
 This follows from that fact that $\mathcal I_N, \mathcal P_N$ are bounded operators on $H^s(\mathbb U)$, $s > 1/2$, and the algebra property for $H^s(\mathbb U)$, $s > 1/2$.
\end{proof}
In our proofs below, we will have an approximation step where $g$ is replaced with $\mathcal P_M g$ and then the resulting error, which will be small, is estimated separately.

Then to analyze \eqref{eq:approx_SIE_P} we need to identify $\mathcal N_{1,N}$ in Theorem~\ref{t:main}.  Since $\mathcal C^-$ commutes with $\mathcal P_N$, \eqref{eq:approx_SIE_P} is equivalent to
\begin{align*}
     (\mathcal C^+ - \mathcal M_N(g)  \mathcal P_N \mathcal C^-) \vec u_N = \mathcal P_N (\vec g - \vec 1), \quad \vec u_N \in \ran \mathcal P_N, \quad \mathcal M_{N}(g) = \mathcal P_N \mathcal M(g) \mathcal P_N.
\end{align*}
While one may be tempted to replace $\mathcal M_N(g)$ with $\mathcal M_{N}(g^{-1})$ to mirror the previous calculations, it is not entirely clear how to analyze the resulting operator.  So, instead, replace $\mathcal M_{N}(g)$ with $\mathcal M_{N}(g)^{-1} $, where the inverse is taken on the invariant subspace $\ran \mathcal P_N$, provided such an inverse exists --- an issue we take up below.  So, supposing this inverse exists, set
\begin{align*}
   \mathcal N_{1,N} = - \mathcal M_{N}(g)^{-1} \mathcal P_N \mathcal C^- + \mathcal P_N \mathcal C^-.
\end{align*}
We then consider the composition
\begin{align*}
    \mathcal J_N := (\mathcal N_0 + \mathcal N_{1,N}) (\mathcal L_0 + \mathcal P_N \mathcal L_1) & = ( \id - (\mathcal M_{N}(g)^{-1} -1)  \mathcal P_N \mathcal C^-)(\id - (\mathcal M_{N}(g) -1)  \mathcal P_N \mathcal C^-)\\
    & = (\id - \mathcal M_N(g)^{-1} \mathcal P_N \mathcal C^- + \mathcal P_N \mathcal C^-) (\id - \mathcal M_N(g) \mathcal P_N \mathcal C^- + \mathcal P_N \mathcal C^-).
\end{align*}
We replace the first $\mathcal C^-$ with $\mathcal C^+ - \id$ giving
\begin{align*}
     \mathcal J_N & = (\id - \mathcal M_N(g)^{-1} \mathcal P_N \mathcal C^+ + \mathcal P_N \mathcal C^+ + \mathcal M_N(g)^{-1} - \mathcal P_N) (\id - \mathcal M_N(g) \mathcal P_N \mathcal C^- + \mathcal P_N \mathcal C^-)\\
     & = \id - \mathcal M_N(g) \mathcal P_N \mathcal C^- + \mathcal P_N \mathcal C^-  - \mathcal M_N(g)^{-1}  \mathcal P_N \mathcal C^+ + \mathcal M_N(g)^{-1}  \mathcal P_N \mathcal C^+\mathcal M_N(g) \mathcal P_N \mathcal C^-\\
     & + \mathcal P_N \mathcal C^+  - \mathcal P_N \mathcal C^+ \mathcal M_N(g) \mathcal P_N \mathcal C^- + \mathcal M_N(g)^{-1} - \mathcal P_N \mathcal C^- + \mathcal M_N(g)^{-1} \mathcal P_N \mathcal C^-\\
     & - \mathcal P_N + \mathcal M_N(g) \mathcal P_N \mathcal C^- - \mathcal P_N \mathcal C^-\\
     & = \id - (\id - \mathcal M_N(g)^{-1} ) \mathcal P_N \mathcal C^+ \mathcal M_N(g)\mathcal P_N \mathcal C^-
\end{align*}
Here we used that $\mathcal C^+\mathcal C^- = 0$.  We identify
\begin{align*}
    \mathcal K_N = (\id - \mathcal M_{N}(g)^{-1}) \mathcal P_N \mathcal C^+ \mathcal M_{N}(g) \mathcal P_N \mathcal C^-,
\end{align*}
and its candidate limit
\begin{align*}
    \mathcal K = (\id - \mathcal M(g^{-1})) \mathcal C^+ \mathcal M(g) \mathcal C^-.
\end{align*}

Now, suppose that $g$ is such that $g_j = 0$ for $j > M$. We compute the matrix representation for $\mathcal C^+ \mathcal M_N(g) \mathcal P_N\mathcal C^-$ similar to Lemma~\ref{l:compact}:
\begin{align*}
    \mathcal C^+ \mathcal M_N(g) \mathcal P_N\mathcal C^- = -\left[ \begin{array}{ccccccc|ccc} 
    &&&&&\ddots &&& \iddots\\
    &&&&& & 0 & 0\\
    \hline
     \cdots & 0 & h_{N_+} & \cdots & \cdots &h_2&h_1 & 0 \\
    \cdots& 0  & h_{N_+ + 1} & h_{N_+} & \cdots &h_3&h_2 & & \ddots\\
   \cdots &0  &h_{N_+ + 2} &\ddots& \ddots & \vdots & \vdots && \\
    & \vdots &\vdots &\ddots&\ddots & h_{N_+} & \vdots\\
    \cdots & 0 & h_{2 N_+}& \cdots & h_{N_++2} &h_{N_++1}  & h_{N_+}\\
    \iddots & 0 & 0 &\cdots& 0 & 0 & 0 \\
     & \vdots & \vdots && \vdots & \vdots & \vdots
   \end{array} \right].
\end{align*}
So, if $N$ is such that $N_+ > M$, then we see that,
\begin{align*}
    \mathcal P_N \mathcal C^+ \mathcal M_N(g) \mathcal P_N\mathcal C^- = \mathcal C^+ \mathcal M(g) \mathcal C^- = \mathcal P_N \mathcal C^+ \mathcal M(g) \mathcal P_N\mathcal C^-.
\end{align*}

For such an $N$, we then have
\begin{align*}
 \mathcal K_N - \mathcal K = ( \mathcal M(g^{-1}) - \mathcal M_N(g)^{-1}) \mathcal P_N \mathcal C^+ \mathcal M(g) \mathcal C^-.
\end{align*}
For a general $g$ we write $g = \mathcal P_M g + (\id - \mathcal P_M g)$ and find
\begin{align*}
 \mathcal K_N - \mathcal K = ( \mathcal M(g^{-1}) - \mathcal M_N(g)^{-1}) \mathcal P_N \mathcal C^+ \mathcal M(\mathcal P_M g) \mathcal C^- + \mathcal E_{N,M},
\end{align*}
where
\begin{align*}
    \mathcal E_{N,M} = (\id - \mathcal M_N(g)^{-1}) \mathcal C^+ (\mathcal M_N( (\id - \mathcal P_M )g )) \mathcal C^- - (\id - \mathcal M(g^{-1})) \mathcal C^+ (\mathcal M( (\id - \mathcal P_M )g )) \mathcal C^-
\end{align*}

And because $\mathcal C^+ \mathcal M(g) \mathcal C^-$ is compact, we have some hope that this will converge to zero in operator norm.  Suppose $s \in \{0\} \cup (1/2, \infty)$ and $g \in H^t$ for $ t > \max\{s,1\}$.  For $s \in (1/2,\infty)$ we have
\begin{align}\label{eq:rh-est}
    \|\mathcal K_N - \mathcal K\|_s \leq \|(\mathcal M(g^{-1}) - \mathcal M_N(g)^{-1}) \mathcal P_N\|_{t \to s} \|\mathcal C^+ \mathcal M(g) \mathcal C^-\|_{s \to t} + \|\mathcal E_{N,M}\|_s,
\end{align}
where $\|\mathcal C^+ \mathcal M(g) \mathcal C^-\|_{s \to t}$ is finite by Lemma~\ref{l:compact}.  For $s > 1/2$, since $H^s(\mathbb U)$ has the algebra property, we estimate
\begin{align}\label{eq:E}
    \|\mathcal E_{N,M}\|_s \leq (2 + \|\mathcal M_N(g)^{-1}\|_s + \|\mathcal M(g^{-1})\|_s) \|(\id - \mathcal P_M) g\|_s.
\end{align}
If $s = 0$, $\|(\id - \mathcal P_M) g\|_s$ must be replaced with the $L^\infty(\mathbb U)$ norm.  And we will see that we can take $M$ proportional to $N$.

Then, to demonstrate that (i) $\mathcal M_N(g)^{-1}$ exists, and (ii) $\|(\mathcal M(g^{-1}) - \mathcal M_N(g)^{-1}) \mathcal P_N\|_{t \to s} \to 0$, we have the following theoretical developments.  The proof of the following is given in \cite[Theorem 2.11]{Bottcher1999}.

\begin{theorem}[Gohberg-Feldman]\label{t:G-F}
Suppose $h$ is continuous and that
\begin{align*}
    \mathcal T(g) := \mathcal C^+ \mathcal M(g) \mathcal C^+
\end{align*}
is boundedly invertible on $\ran \mathcal C^+$ in the $L^2(\mathbb U)$ norm.  Then $\mathcal M_N(g)$ is invertible for sufficiently large $N > N_0$ and
\begin{align*}
    \sup_{N > N_0} \|\mathcal  M_N(g)^{-1}\|_0 < \infty.
\end{align*}
\end{theorem}

A sufficient condition for $\mathcal C^+ \mathcal M(g) \mathcal C^+$ to be invertible directly relates to the solvability of the problem under consideration and comes from the following proposition that can be found in \cite[Proposition 1.13]{Bottcher1999}.  First, we define the Hardy space on the disk $\mathbb D = \{z \in \mathbb C : |z| < 1\}$:
\begin{align*}
    H^\infty(\mathbb D) := \left\{ f : \mathbb D \to \mathbb C, \text{ analytic} : \sup_{z \in \mathbb D} |f(z)| < \infty\right\}.
\end{align*}
The Hardy space $\overline{H^\infty(\mathbb D)}$ is the set of all functions $f$ such that $g(z) = f(1/z)$ satisfies $g \in H^\infty(\mathbb D)$.

\begin{proposition}\label{p:prod}
Suppose $a \in \overline{H^\infty(\mathbb D)}$, $b \in L^\infty(\mathbb U)$ and $c \in {H^\infty}(\mathbb D)$, then
\begin{align*}
    \mathcal T(abc) = \mathcal T(a) \mathcal T(b) \mathcal T(c).
\end{align*}
\end{proposition}

We arrive at the corollary that we require.
\begin{corollary}\label{cor:inv}
Suppose the RH problem \eqref{eq:rhp}, with $g \in H^1(\mathbb U)$, $\min_{z \in \mathbb U} |g(z)| > 0$, has a unique solution of the form \eqref{eq:solnform} where $u \in H^1(\mathbb U)$.  Then $\mathcal T(g)$ is invertible, and therefore there exists $N_0 > 0$ such that
\begin{align*}
    \sup_{N > N_0} \|\mathcal  M_N(g)^{-1}\|_0 < \infty.
\end{align*}
\end{corollary}
\begin{proof}
If $u \in H^1(\mathbb U)$ then the solution extends continuously up to $\mathbb U$ from both the interior and exterior of $\mathbb U$.  Now, if $\phi$ vanished, say $\phi(z^*)= 0$ for $|z^*| \neq 1$ then it follows that $\phi(z) ( 1 + 1/(z-z^*))$ is another solution, contradicting uniqueness. So if $\phi$ does vanish it has to be on the boundary $\mathbb U$.   Suppose $\phi^+(z^*) = 0$, Then the jump condition
\begin{align*}
    0 = \phi^-(z^*) g(z^*).
\end{align*}
This implies that $\phi^-(z^*) = 0$.  Then let $\gamma$ be a small subarc of $\mathbb U$ that contains $z^*$ in its interior.  Consider the functions
\begin{align*}
    \mathfrak g(z) &= \exp\left( \frac{1}{ 2 \pi \ii} \int_\gamma \frac{\log g(z')}{z' - z} \dd z \right),\\
    \tilde \phi(z) & = \phi(z)/\mathfrak g(z),
\end{align*}
where the branch of the logarithm is taken so that $\log g$ is continuous on $\gamma$ (shrinking $\gamma$ if necessary).  Now, it follows that $\tilde \phi(z)$ has no jump on $\gamma$:
\begin{align*}
     \tilde \phi^+(z) = \tilde \phi^-(z), \quad z \in \gamma.  
\end{align*}
And then because $\mathfrak g(z)$ does not vanish near $z^*$, we conclude that $\tilde \phi(z^*) = 0$, implying that $\phi(z) ( 1 + 1/(z-z^*))$ is another solution as it has sufficiently smooth boundary values near $z = z^*$.  Thus, we have shown that $(\phi^-)^{-1} \in \overline{H^\infty(\mathbb U)}$ and Proposition~\ref{p:prod} demonstrates that
\begin{align*}
    \mathcal T(g)^{-1} = \mathcal T(\phi^+) \mathcal T((\phi^-)^{-1}),
\end{align*}
and an application of Theorem~\ref{t:G-F} completes the proof.
\end{proof}

With this in hand, we consider,
\begin{align*}
    \mathcal M(g) v = \mathcal P_N f, \quad \mathcal M_N(g) v_N = \mathcal P_N f.
\end{align*}
Thus
\begin{align*}
     v_N - v &= (\mathcal M_N(g)^{-1} - \mathcal M(g^{-1})) \mathcal P_N f,\\
     &= (\mathcal M_N(g)^{-1}\mathcal M(g)\mathcal M(g^{-1}) - \mathcal M(g^{-1})) \mathcal P_N f\\
     & = (\mathcal M_N(g)^{-1}\mathcal M(g) - \id) \mathcal M(g^{-1})\mathcal P_N f\\
     & = (\mathcal M_N(g)^{-1}\mathcal M(g) - \mathcal P_N) \mathcal M(g^{-1})\mathcal P_N f + (\mathcal P_N - \id) \mathcal M(g^{-1})\mathcal P_N f\\
     & = \mathcal M_N(g)^{-1}(\mathcal M(g) - \mathcal M_N(g)) \mathcal M(g^{-1})\mathcal P_N f + (\mathcal P_N - \id) \mathcal M(g^{-1})\mathcal P_N f
\end{align*}
Then we consider, for $N > N_0$
\begin{align*}
    \|v_N - v\|_s \leq \left[ \sup_{N > N_0} \|\mathcal  M_N(g)^{-1}\|_{s} \| \mathcal M(g) - \mathcal M_N(g)\|_{t \to s}   + \|\id - \mathcal P_N\|_{t \to s} \right]\|\mathcal M(g^{-1})\|_t \|f\|_t.
\end{align*}
Then, we estimate
\begin{align*}
    \| \mathcal M(g) - \mathcal M_N(g)\|_{t \to s} &= \| \mathcal M(g) - \mathcal P_N \mathcal M(g) \mathcal P_N\|_{t \to s} \\
    &\leq \| (\id - \mathcal P_N)\mathcal M(g)\|_{t \to s} + \| \mathcal P_N \mathcal M(g) (\id -  \mathcal P_N)\|_{t \to s}\\
    & \leq \|\id - \mathcal P_N\|_{t \to s} ( \|\mathcal M(g)\|_t + \|\mathcal M(g)\|_s).
\end{align*}
So,
\begin{align*}
    \|v_N - v\|_s \leq \| \id - \mathcal P_N\|_{t \to s} \left[ \sup_{N > N_0} \|\mathcal  M_N(g)^{-1}\|_{s} ( \|\mathcal M(g)\|_t + \|\mathcal M(g)\|_s) + 1 \right] \|\mathcal M(g^{-1})\|_t \|f\|_t.
\end{align*}
Then we must estimate, for $s \geq 0$, $N > N_0$,
\begin{align*}
    \|\mathcal  M_N(g)^{-1}\|_{s} = \max_{u \in \ran \mathcal P_N}\frac{\|\mathcal  M_N(g)^{-1}u\|_{s}}{\|u\|_{s}} \leq \max_{u \in \ran \mathcal P_N}\frac{N^s\|\mathcal  M_N(g)^{-1}u\|_{0}}{\|u\|_{0}},
\end{align*}
giving the final estimate
\begin{align*}
    \|v_N - v\|_s \leq A_{s,t} N^{2s -t} \|f\|_t,
\end{align*}
which implies
\begin{align}\label{eq:mult-diff}
    \|(\mathcal M(g^{-1}) - \mathcal M_N(g)^{-1}) \mathcal P_N\|_{t \to s} \leq A_{s,t} N^{2s -t}.
\end{align}
And since $g \in H^t(\mathbb U)$, $t > \max\{s,1\}$ and $g$ does not vanish we have for $s > 1/2$
\begin{align*}
    \|\mathcal E_{N,M}\|_s \leq (2 + C_1 N^s + C_2 ) \|(\id - \mathcal P_M) g\|_s \leq  (2 + C_1 N^s + C_2 ) D_{s,t} N^{s-t} \|g\|_t,
\end{align*}
by Theorem~\ref{t:proj-Hs} (for constants $C_1,C_2$).  And for $s = 0$
\begin{align*}
    \|\mathcal E_{N,M}\|_0 \leq (2 + C_1 + C_2 ) \|(\id - \mathcal P_M) g\|_{L^\infty(\mathbb U)} \leq  (2 + C_1 + C_2 ) D_{r,t} N^{r-t} \|g\|_t,
\end{align*}
again by Theorem~\ref{t:proj-Hs} (for different constants $C_1,C_2$ and $1/2 < r < t$).  Using \eqref{eq:rh-est}, Theorem~\ref{t:main} applies giving the following result for $2s < t$.

\begin{theorem}\label{t:rhp-1}
Suppose $g \in H^t(\mathbb U)$, $t \geq 1$, with $\min_{z \in \mathbb U}|g(z)| > 0$.  Assume RH problem \eqref{eq:rhp} has a unique solution $u$ of the form \eqref{eq:solnform}, $u \in H^t(\mathbb U)$.  Then, for  $s \in \{0\} \cup (1/2, \infty)$ satisfying $0 \leq 2s < t$, there exists $N_0 > 0$ such that the operator
\begin{align*}
    \mathcal C^+ - \mathcal M_N(g) \mathcal C^- \quad \text{on} \quad \ran \mathcal P_N,
\end{align*}
is invertible and the unique solution $u_N$ of
\begin{align*}
    (\mathcal C^+ - \mathcal M_N(g) \mathcal C^-) u_N = \mathcal P_N(  g -  1),
\end{align*}
satisfies
\begin{align*}
    \|u - u_N\|_s  \leq C_{s,t} N^{s - t}.
\end{align*}
\end{theorem}

While this is an important step, constructing a matrix representation for the operator $\mathcal C^+ - \mathcal M_N(g) \mathcal C^-$ is not as easy as it is to work with the interpolation projection $\mathcal I_N$. So, consider the equation
\begin{align*}
    (\id - \mathcal I_N \mathcal M(g-1) \mathcal C^-) u_N = \mathcal I_N(  g -  1), \quad u_N \in \ran \mathcal P_N.
\end{align*}
Now, we look to apply Theorem~\ref{t:IN-PN}, and to do this, we need to show that $\mathcal C^+ - \mathcal I_N \mathcal M(g) \mathcal C^-$ is invertible on $\ran \mathcal P_N$.  

By Theorem~\ref{t:rhp-1} know that for sufficiently large $N$ if $g$ is sufficiently regular and $N$ is sufficiently large, the operator
\begin{align*}
    \id - \mathcal P_N \mathcal M(g-1) \mathcal C^-,
\end{align*}
is invertible on $\ran \mathcal P_N$.  So, we write
\begin{align*}
    \id - \mathcal I_N \mathcal M(g-1) \mathcal C^- = \id - \mathcal P_N \mathcal M(g-1) \mathcal C^- + (\mathcal P_N - \mathcal I_N) \mathcal M(g-1) \mathcal C^-.
\end{align*}
To show this is invertible, we pre-multiply this by three operators and show that the resulting operator is approximately lower-triangular when restricted to $\ran \mathcal P_N$.
The following is helpful in determining when inverse operators exist.

\begin{lemma}
Suppose the RH problem \eqref{eq:rhp}, with $g \in H^t(\mathbb U)$, $t \geq 1$, $\min_{z \in \mathbb U} |g(z)| > 0$ has a unique solution of the form \eqref{eq:solnform} where $u \in H^t(\mathbb U)$.  Then\footnote{To ease notation, we write $\phi_\pm^{-1}$ in place of $(\phi^\pm)^{-1}$.}
\begin{align*}
    (\id - \mathcal M(g-1) \mathcal C^-)^{-1} = \mathcal M(\phi^+ ) \mathcal C^+ \mathcal M( \phi_+^{-1} ) - \mathcal M(\phi^- ) \mathcal C^- \mathcal M( \phi_+^{-1} ).
\end{align*}
\end{lemma}
\begin{proof}
The proof follows from a direct calculation where one needs to note that, following the proof of Corollary~\ref{cor:inv}, $\phi$ vanishes nowhere, and therefore $\phi_\pm^{-1} \in H^t(\mathbb U)$.  
\end{proof}

Using \eqref{eq:E} and \eqref{eq:mult-diff}, we have that
\begin{align*}
    (\mathcal C^+ - \mathcal M(g^{-1}) \mathcal C^-)(\mathcal C^+ - \mathcal M(g) \mathcal C^-)  - (\id + \mathcal P_N \mathcal C^- - \mathcal M_N(g)^{-1} \mathcal C^-)(\id  - \mathcal M_N (g-1) \mathcal C^-) \to 0,
\end{align*}
in the $H^s(\mathbb U)$ operator norm, if $g \in H^t(\mathbb U)$ and $t > 2s$.  Thus
\begin{align*}
    \underbrace{(\mathcal C^+ - \mathcal M(g) \mathcal C^-)^{-1}(\mathcal C^+ - \mathcal M(g^{-1}) \mathcal C^-)^{-1} (\id + \mathcal P_N \mathcal C^- - \mathcal M_N(g)^{-1} \mathcal C^-)}_{\mathcal S_N}(\id  - \mathcal M_N (g-1) \mathcal C^-) \to \id,
\end{align*}
again in the $H^s(\mathbb U)$ operator norm.  We then observe something important about $\mathcal S_N$.  If $v = \mathcal C^+ u$, $u \in \ran \mathcal P_N$, then $\mathcal S_N v = v$.  This follows because $\mathcal C^- v = 0$, $\mathcal C^+ v = v$. The next task is to then show that
\begin{align}\label{eq:ran-cond}
    \ran (\mathcal P_N - \mathcal I_N) \mathcal M(g-1) \mathcal C^-\mathcal P_N \subset \ran \mathcal C^+.
\end{align}

To assist in this, let $I_M$ be the $M \times M$ identity matrix and we set $I_+ = I_{N_+ +1}$, $I_- = I_{N_-}$.  Then the matrix form for $\mathcal P_N: H^s(\mathbb U) \to H^s(\mathbb U)$ is given, in block form, by
\begin{align*}
\mathcal P_N = \left[ \begin{array}{ccc|cccccc}
    \ddots & \ddots & \ddots & & & \iddots \\
    \ddots &  0 & 0 & 0 \\
    \ddots & 0 & I_- & 0 & 0 \\ \hline
    &0 & 0 & I_+ & 0 & \ddots \\
    &&0 & 0 & 0 & \ddots \\
   \iddots &&& \ddots & \ddots & \ddots
    \end{array} \right],
\end{align*}
and considering $\mathcal P_N : H^s(\mathbb U) \to \ran \mathcal P_N$
\begin{align*}
    \mathcal P_N = \left[ \begin{array}{cccc|cccccc}
    \cdots & 0  & 0 & I_- & 0 & 0 & 0 &\cdots\\ \hline
    \cdots &0 & 0 & 0 & I_+ & 0 & 0 &  \cdots \\
    \end{array} \right],
\end{align*}
We then write $\mathcal I_N: H^s(\mathbb U) \to \ran \mathcal P_N$,  $s > 1/2$, as 
\begin{align*}
    \mathcal I_N = \left[ \begin{array}{cccc|cccccc}
    \cdots & I_-  & 0 & I_- & 0 & I_- & 0 &\cdots\\ \hline
    \cdots &0 & I_+ & 0 & I_+ & 0 & I_+ &  \cdots \\
    \end{array} \right],
\end{align*}
and therefore, as acting on $H^s(\mathbb U)$,
\begin{align*}
\mathcal I_N - \mathcal P_N = \left[ \begin{array}{cccc|cccccc}
    \cdots & I_-  & 0 & 0 & 0 & I_- & 0 &\cdots\\ \hline
    \cdots &0 & I_+ & 0 & 0 & 0 & I_+ &  \cdots \\
    \end{array} \right].
\end{align*}
Write $h = g -1$, and suppose that $\mathcal P_N g = g$ so that
as an operator from $\ran \mathcal P_N$ to $H^s(\mathbb U)$
\begin{align*}
  \mathcal M(g-1) \mathcal C^- \mathcal P_N = -\left[\begin{array}{ccc|ccccccc}
    \vdots & &\vdots & \\
     0& &  \vdots\\ 
     h_{-N_-} & \ddots & \vdots  \\
     \vdots &\ddots & 0 && \iddots \\
      h_{-1} & \ddots & h_{-N_-} & 0 \\ 
     h_0 & \ddots & \vdots & 0  \\
     \vdots & \ddots & h_{-1} && \ddots \\ 
  h_{N_+} &  & h_0  \\ \hline
 0  & \ddots & \vdots \\
  \vdots& \ddots & h_{N_+} \\
   \vdots & & 0\\
   \vdots &  & \vdots
    \end{array} \right].
\end{align*}
Define the lower-triangular matrix $\mathcal A_N(h)$ by
\begin{align*}
    \mathcal A_N(h) = \begin{bmatrix} h_{-N_-} &  &   \\
     \vdots &\ddots &  \\
      h_{-1} & \cdots & h_{-N_-}  \end{bmatrix},
\end{align*}
if $N_- = N_+ + 1$. If $N_- = N_+$ set
\begin{align*}
    \mathcal A_N(h) = \begin{bmatrix} 0 & \cdots & 0 \\
    h_{-N_-} &  &   \\
     \vdots &\ddots &  \\
      h_{-1} & \cdots & h_{-N_-}  \end{bmatrix}.
\end{align*}

And then we can compute, as an operator on $\ran \mathcal P_N$,
\begin{align}\label{eq:lower-tri-op}
    (\mathcal I_N - \mathcal P_N)\mathcal M(g-1) \mathcal C^+ \mathcal P_N = \left[\begin{array}{c|c}
    0 & 0 \\ \hline
    - \mathcal A_N(h) & 0
     \end{array} \right].
\end{align}
This establishes \eqref{eq:ran-cond}.  Furthermore, this establishes that
\begin{align}\label{eq:invert-op}
    \id - (\mathcal I_N - \mathcal P_N)\mathcal M(g-1) \mathcal C^+ \mathcal P_N,
\end{align}
is invertible on $\ran \mathcal P_N$.  Then we consider as an operator on $\ran \mathcal P_N$
\begin{align*}
    \mathcal P_N \mathcal S_N ( \id - \mathcal I_N \mathcal M(g-1) \mathcal C^-) \mathcal P_N = \mathcal P_N \mathcal S_N ( \id - \mathcal P_N \mathcal M(g-1) \mathcal C^-) \mathcal P_N - (\mathcal I_N - \mathcal P_N)\mathcal M(g-1) \mathcal C^+\mathcal P_N.
\end{align*}
This, in operator norm, is close to the invertible operator \eqref{eq:invert-op} and to show that this operator is invertible, for $N$ sufficiently large, we need a bound on the inverse of \eqref{eq:invert-op}.  But this follows because \eqref{eq:lower-tri-op} is bounded on $H^s(\mathbb U)$, $s > 1/2$, and on $\ran \mathcal P_N$
\begin{align*}
    (\id - (\mathcal I_N - \mathcal P_N)\mathcal M(g-1) \mathcal C^+ \mathcal P_N)^{-1} = \id + (\mathcal I_N - \mathcal P_N)\mathcal M(g-1) \mathcal C^+ \mathcal P_N.
\end{align*}
The general case where $\mathcal P_N g \neq g$ can then be treated using Lemma~\ref{l:truncation}.   We arrive at the following result.  

\begin{theorem}
Theorem~\ref{t:rhp-1} holds with $\mathcal P_N$ replaced with $\mathcal I_N$ and a, potentially, different constant $C_{s,t}$. 
\end{theorem}

\begin{remark}
In many relevant applications, in the Riemann--Hilbert problem one has $g \in \mathbb C^{n \times n}$ and $\phi$ can be taken to be either vector or matrix valued.  The extension of the theory used above (e.g., Theorem~\ref{t:G-F}) exists, see \cite[Section 6.2]{Bottcher1999} and it is much more involved.  But the convergence theorems established here extend to this matrix case using this theory. 
 A full discussion of the details here will be left for a future publication.
\end{remark}

\section{Numerical demonstrations}
\subsection{Solving differential equations}
Consider solving
\begin{align}\label{eq:3rdop}
    - \frac{\dd^3 u}{\dd \theta^3} (\theta) + g(\theta) u(\theta) = h(\theta),
\end{align}
where $g(\theta)$ and $h(\theta)$ are given by their Fourier coefficients
\begin{align}
    g(\theta) &= \sum_{j=-\infty}^\infty g_j \ee^{\ii j \theta}, \quad g_j =(1 + |j|)^{-\alpha},\label{eq:g}\\
    h(\theta) &= \sum_{j=-\infty}^\infty h_j \ee^{\ii j \theta}, \quad h_j = \begin{cases} 1 & j = 0,\\
    \mathrm{sign}(j)(1 + |j|)^{-\alpha} & \text{otherwise},\end{cases}\label{eq:h}
\end{align}
and we choose $\alpha > t + 1/2$ ($\alpha = 1.51$) so that $h,g \in H^t(\mathbb T).$    Then it follows that the solution $u$ of this equation satisfies $u \in H^{t + 3}(\mathbb T)$.  In Figure~\ref{f:diff-error} we demonstrate that the convergence rate of the finite-section method applied to this equation converges in $H^s(\mathbb T)$ at effectively the rate $N^{s - t -3}$.
\begin{figure}[tbp]
\includegraphics[width=.8\linewidth]{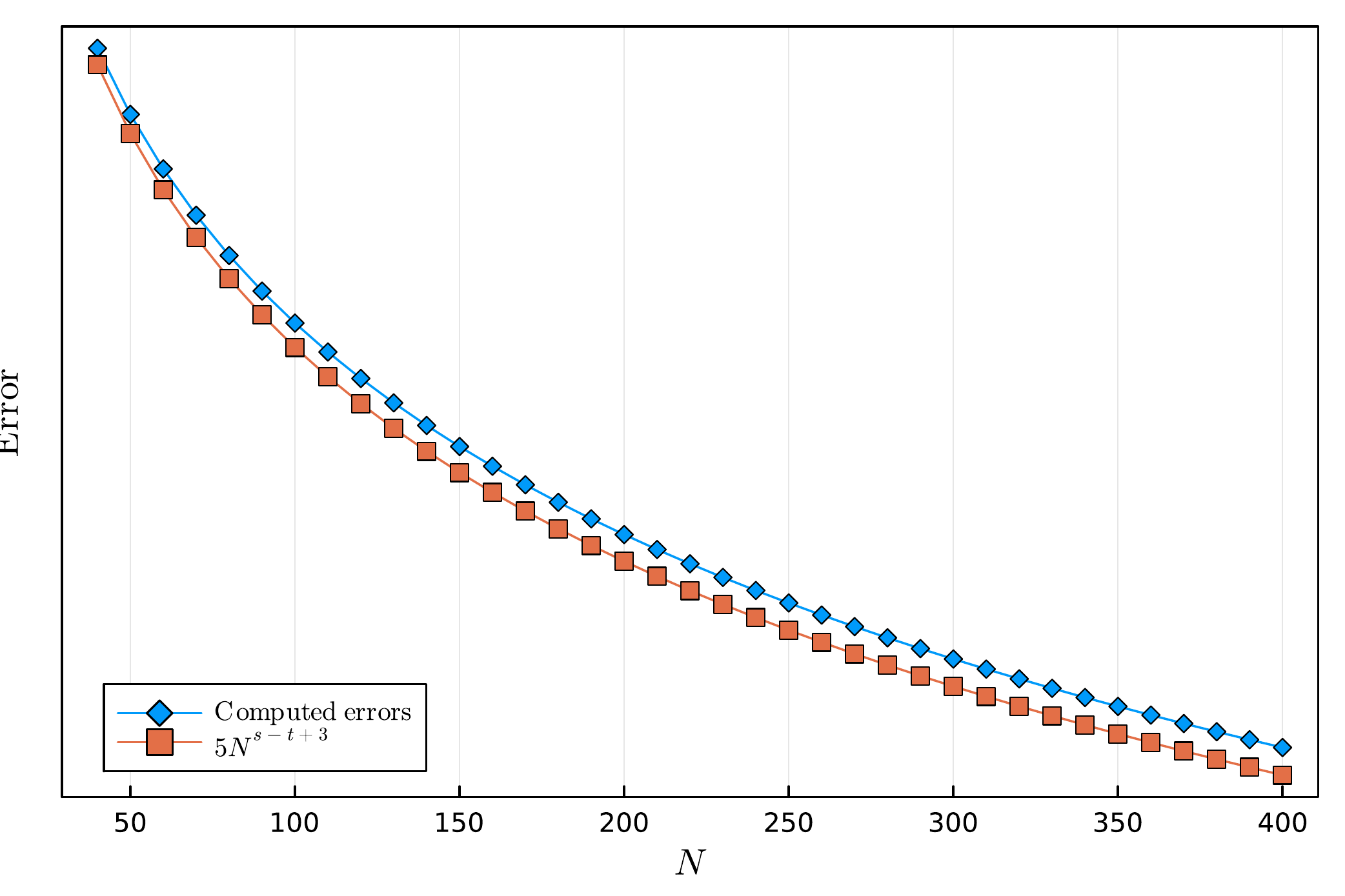}
\caption{\label{f:diff-error}  The measured $H^s(\mathbb U)$ error in solving the euquation \eqref{eq:3rdop}  The diamonds (blue) give the computed error by comparing against a solution computed using $N = 2001.$  The rate of convergence closely matches $N^{s-t + 3}$ reflecting the fact that that Theorem~\ref{t:diff-conv} is optimal.}
\end{figure}

\subsection{Spectrum approximation}

We also consider the demonstration of Theorem~\ref{t:spectrum-main}.  Maybe surprisingly, to really see the true rate of convergence we have to resort to high-precision arithmetic and for this reason we only give examples that are self-adjoint where variable precision eigensolvers are available.

\subsubsection{A second-order operator}

We consider computing the eigenvalues of the operator
\begin{align}\label{eq:2op}
\mathcal L_0 + \mathcal L_1 = - \frac{\dd^2}{\dd \theta^2} + g(\theta),
\end{align}
where $g$ is given in \eqref{eq:g}.  Using $\alpha = 2.51$, giving $\ell = 2, k = 2, p= 0$, the estimates in Theorem~\ref{t:spectrum-main} give
\begin{align*}
    |\lambda_N - \lambda| \leq C (2 + |\lambda|)^{\ell/k} N^{-t},
\end{align*}
where we have dropped the ceiling function in the exponent and have ignored the $N$ \emph{sufficiently large} requirement.  

To analyze the error of the method as $\lambda$ varies, we compute the eigenvalues of $\mathcal L_0 + \mathcal P_N \mathcal L_1$ with $N = 501$ and treat it as the ground truth, and then for $N = 41, 81, 161, 321$, to each approximate eigenvalue $\lambda_j^{(N)}$ we assign a distance (error)
\begin{align*}
    d^{(N)}_j = \min_i | \lambda_j^{(N)} - \lambda_i^{(501)}|. 
\end{align*}
We also define a rescaled error
\begin{align*}
    r^{(N)}_j = d^{(N)}_j N^{\ell} (2 + |\lambda|)^{-\ell/k}.  
\end{align*}
These computations are summarized in Figure~\ref{f:second}.
\begin{figure}[tbp]
\includegraphics[width=.49\linewidth]{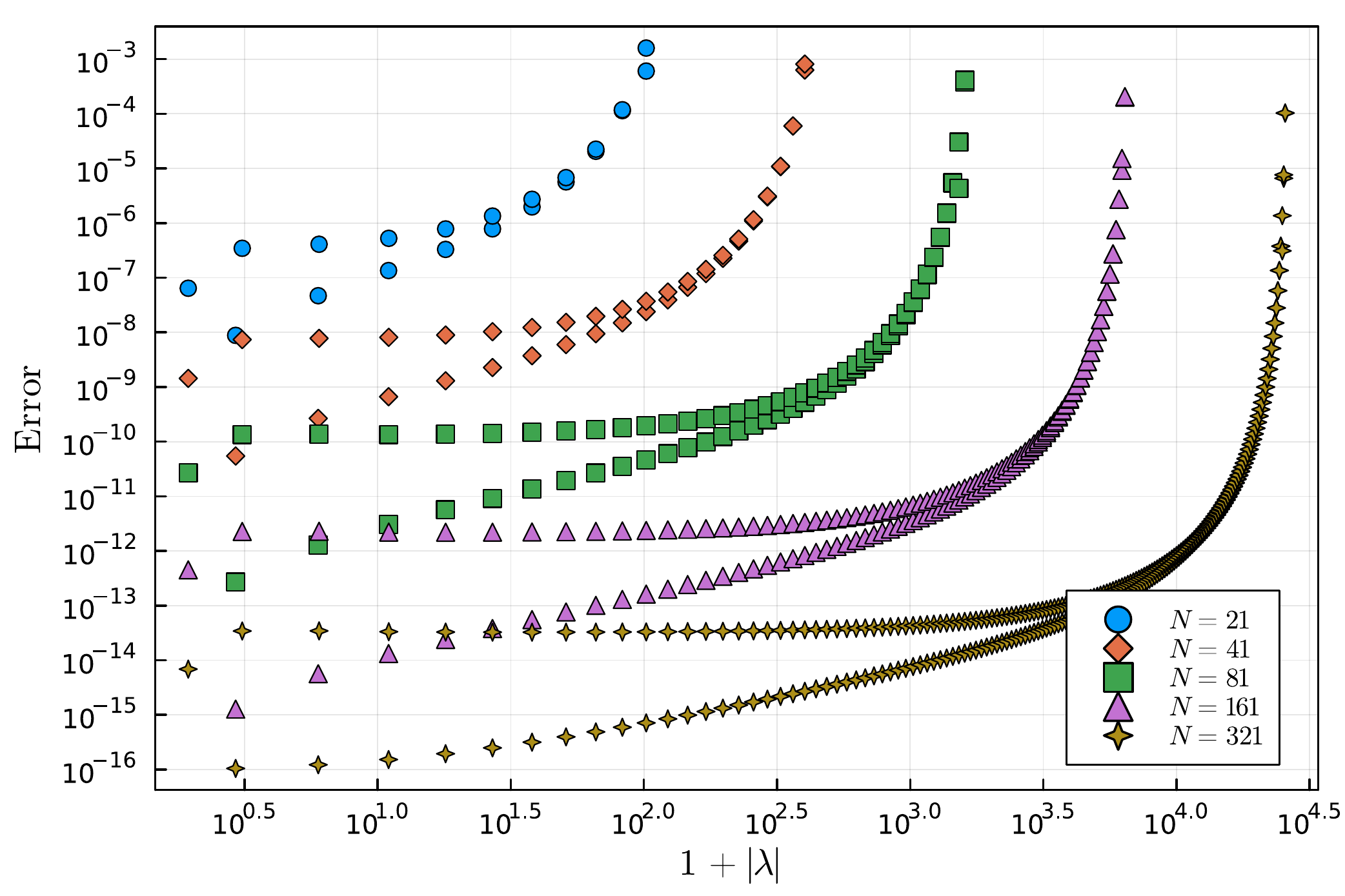}
\includegraphics[width=.49\linewidth]{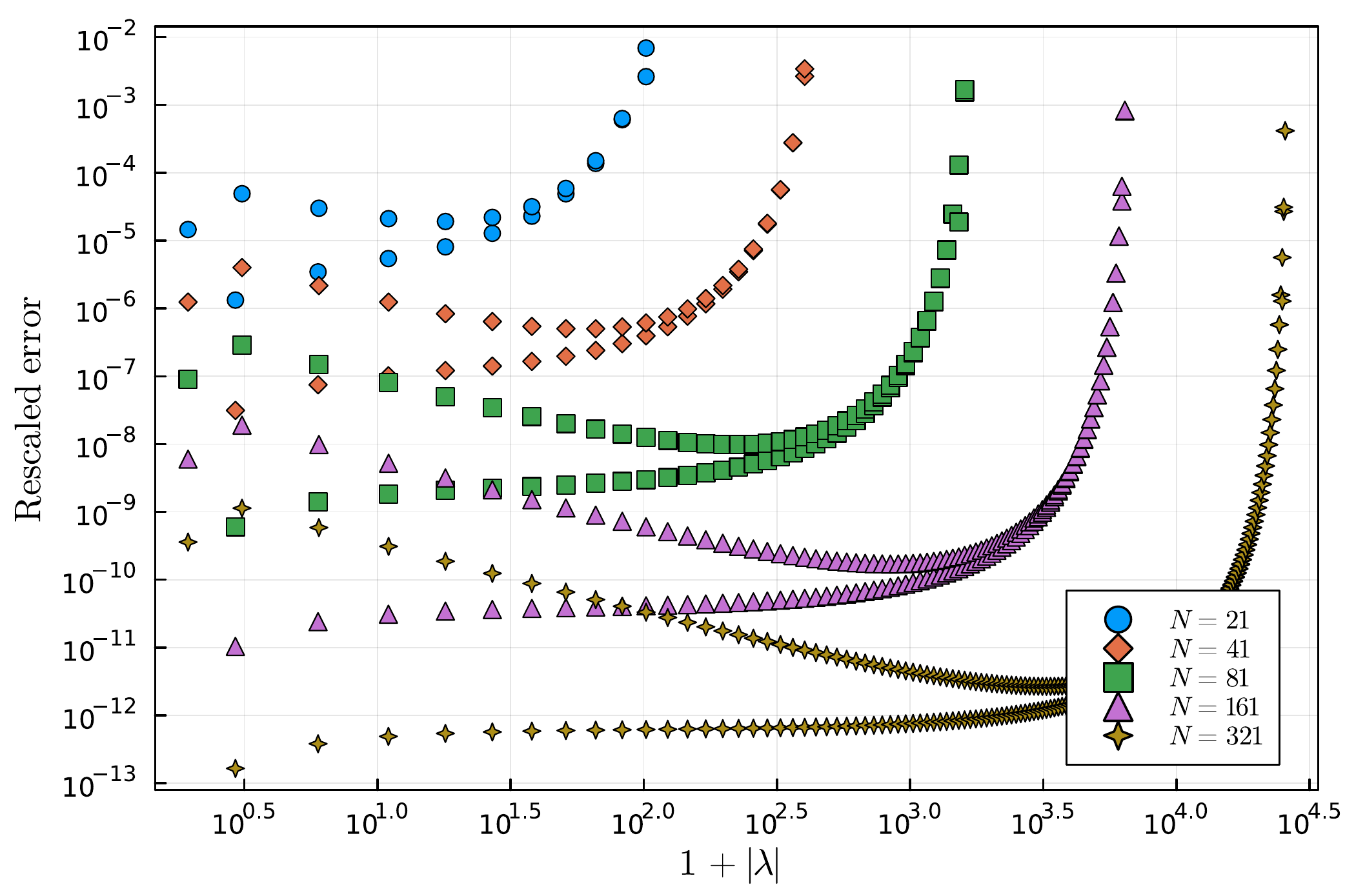}
\caption{\label{f:second}  The errors $d_j^{(N)}$ (left panel) and the rescaled errors $r_j^{(N)}$ plotted versus $1 + |\lambda_j|$ for the finite $N$ approximation of \eqref{eq:2}.  The rescaled errors verify the estimates in Theorem~\ref{t:spectrum-main}.}
\end{figure}

\subsubsection{A third-order operator}
We consider computing the eigenvalues of the operator
\begin{align}\label{eq:3op}
\mathcal L_0 + \mathcal L_1 = - \ii \frac{\dd^3}{\dd \theta^3} + g(\theta),
\end{align}
where \eqref{eq:g}.  Using $\alpha = 2.51$, giving $\ell = 2, k = 3, p= 0$, the estimates in Theorem~\ref{t:spectrum-main} are then summarized similarly as above in Figure~\ref{f:third}.
\begin{figure}[tbp]
\includegraphics[width=.49\linewidth]{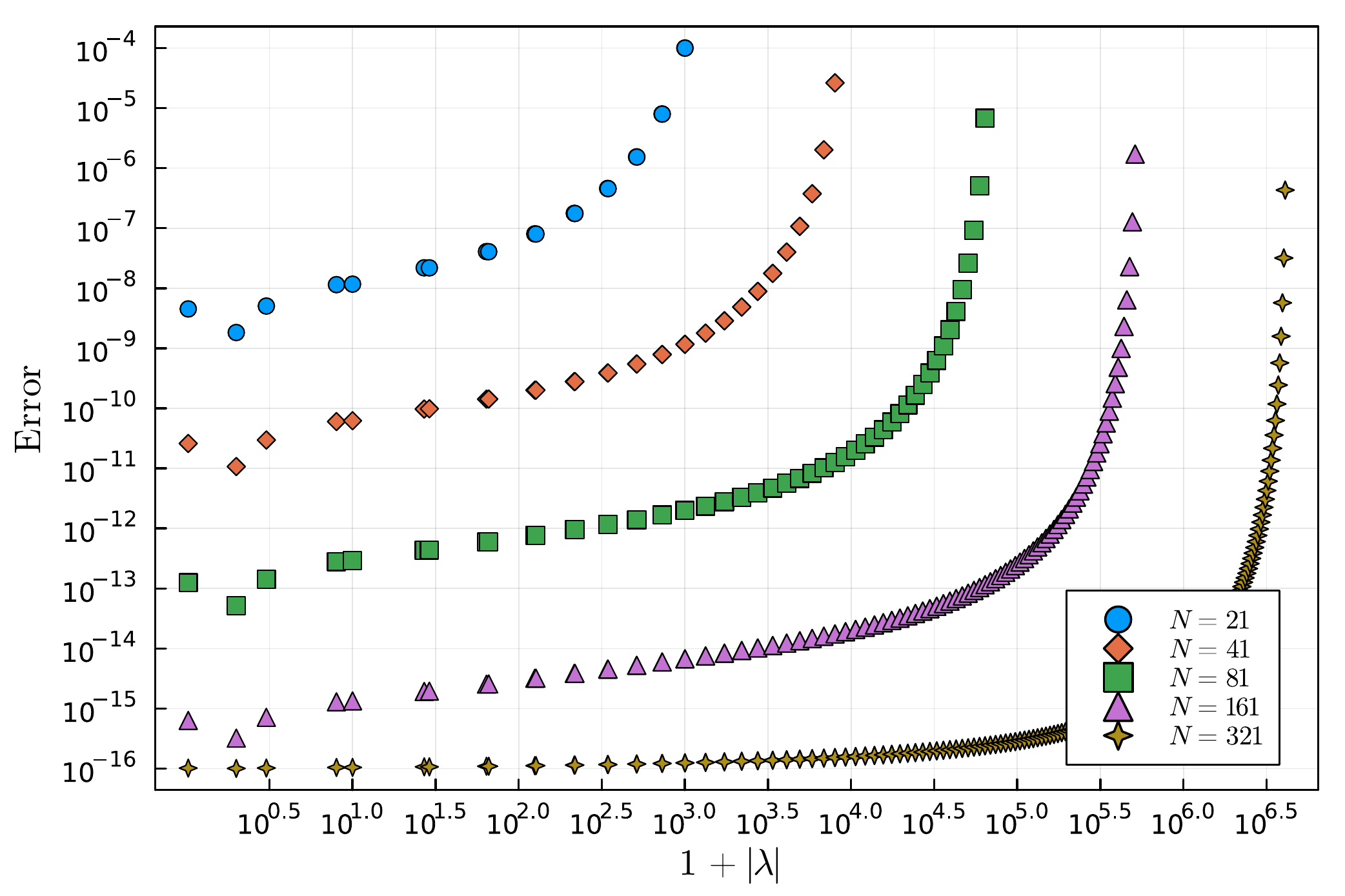}
\includegraphics[width=.49\linewidth]{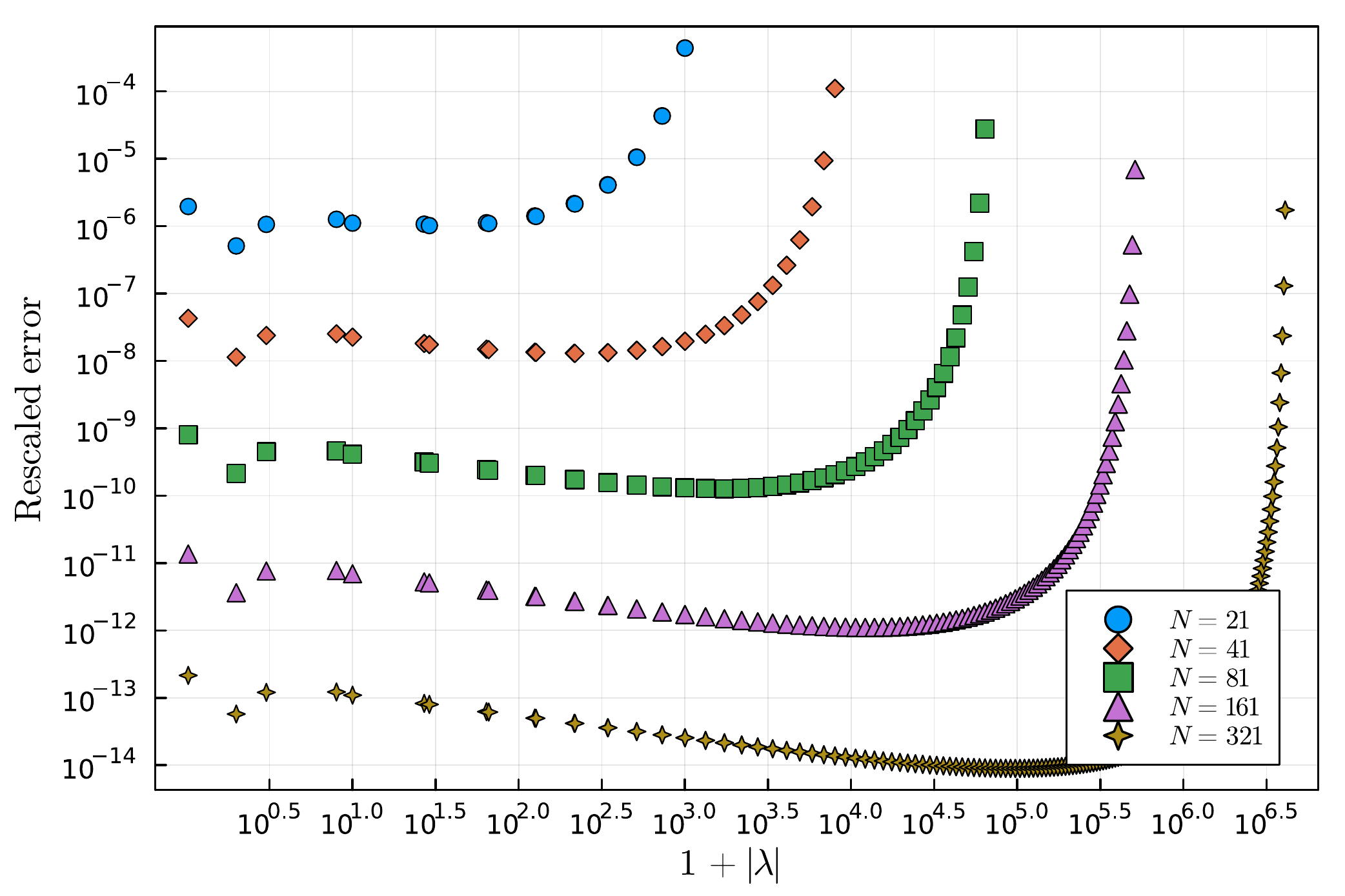}
\caption{\label{f:third}  The errors $d_j^{(N)}$ (left panel) and the rescaled errors $r_j^{(N)}$ plotted versus $1 + |\lambda_j|$ for the finite $N$ approximation of \eqref{eq:3op}.  The rescaled errors verify the estimates in Theorem~\ref{t:spectrum-main}.}
\end{figure}

\subsection{A Riemann--Hilbert problem}

Let $\epsilon > 0$ be small.  Consider the function
\begin{align}\label{eq:gg}
    g(z) = \sum_{j=-\infty}^\infty g_j z^j, \quad g_j = \begin{cases} 1 & j = 0,\\
    \epsilon (1 + |j|)^{-\alpha} & \text{otherwise}.\end{cases}
\end{align}
Thus if $\alpha > t + 1/2$ then $g \in H^t(\mathbb U).$  For this problem we solve \eqref{eq:approx_SIE_P} and compute plot the $H^s(\mathbb U)$ error between $u_N$ and $U_{2000}$ as $N$ varies between $40$ and $400$.  We choose $\alpha = t + 0.51$, $ t = 1$, $\epsilon = 0.01$ and $s = 1/4$.  In Figure~\ref{f:rhp}, we see that our bound is realized.  But we do expect that the condition $t > 2s$ to be unnecessary in Theorem~\ref{t:rhp-1} with $t > s$ being the true requirement.  Note that Theorem~\ref{t:rhp-1} does not apply for $s =1/4$, but we also expect this to be an unnecessary requirement.

\begin{figure}[tbp]
\includegraphics[width=.8\linewidth]{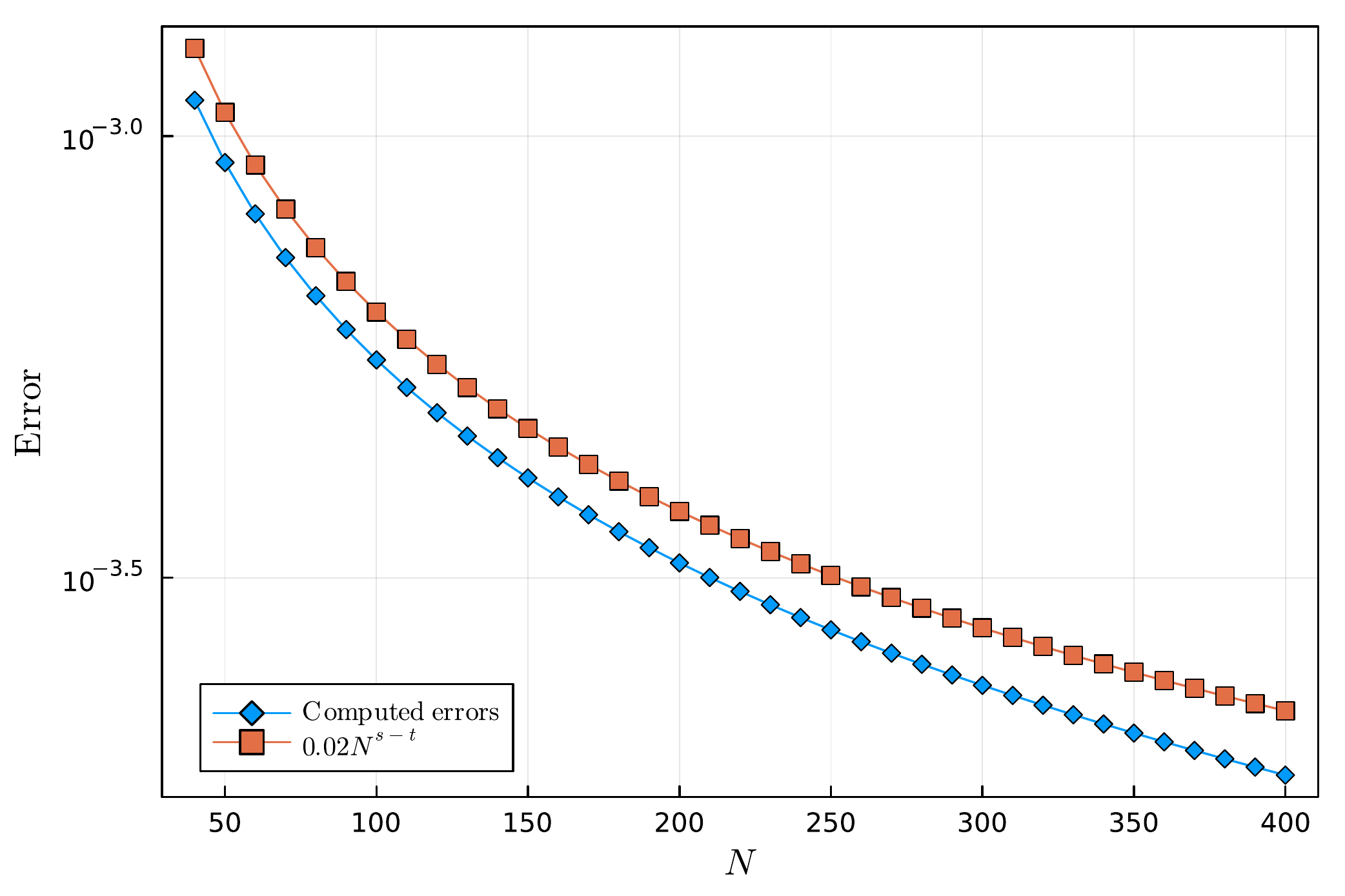}
\caption{\label{f:rhp}  The measured $H^s(\mathbb U)$ error in solving the Riemann--Hilbert problem on the circle with jump condition given by \eqref{eq:gg}.  The diamonds (blue) give the computed error by comparing against a solution computed using $N = 2000.$  The rate of convergence closely matches $N^{s-t}$ indicating that Theorem~\ref{t:rhp-1} is optimal with respect to the rate it predicts.}
\end{figure}

\section{Conclusion and outlook}

The main focus of the current work is to begin filling gaps in the literature so that one can prove convergence of widely-used spectral methods effectively and give rates of convergence.  Our results on differential operators improve that of \cite{Curtis} and \cite{Zumbrun} each in some respects.  

A number of open questions exist.  One problem is to obtain optimal rates in Theorem~\ref{t:spectrum-main} and to extend our analysis to differential operators where the leading coefficients are not constants.  One such approach in this direction would see the Riemann--Hilbert theory connect with the differential operator theory using the result of Gohberg-Feldman, Theorem~\ref{t:G-F}.  Similarly, as we have observed, Theorem~\ref{t:rhp-1} is suboptimal.

The main challenge is to apply this framework in settings that are not as structured -- when orthogonal polynomial bases are used instead of Fourier.  Such settings can be found in \cite{SORHFramework,TrogdonSOKdV,Bilman2022,Ballew2023}.

\bibliographystyle{abbrv}
\bibliography{library}

\end{document}